%% file: free_CD_cats.tex
\documentclass[reqno,11pt]{amsproc}
\usepackage{geometry}
\usepackage[utf8]{inputenc}
\usepackage[dvipsnames]{xcolor}
\usepackage[english]{babel}

\usepackage{verbatim}
\usepackage{amsmath}
\usepackage{amsfonts}
\usepackage{amssymb}
\usepackage{amsthm}
\usepackage{aliascnt}
\usepackage{amsmath}
\usepackage{tikz-cd}
\usepackage{mathdots}
\usepackage{cancel}
\usepackage{color}
\usepackage{stmaryrd}
\usepackage{geometry}
\usepackage{microtype}
\usepackage{mathtools}
\usepackage{textcomp}
\usepackage{enumitem}
\usepackage{caption}
\usepackage{cite}
\usepackage[labelfont=rm]{subcaption}
\usetikzlibrary{shapes.geometric}

\definecolor{myurlcolor}{rgb}{0,0,0.3}
\definecolor{mycitecolor}{rgb}{0,0.4,0}
\definecolor{myrefcolor}{rgb}{0.4,0,0}
\usepackage[pagebackref,draft=false]{hyperref}
\hypersetup{colorlinks,
linkcolor=myrefcolor,
citecolor=mycitecolor,
urlcolor=myurlcolor}

\usepackage[capitalize]{cleveref}

\usepackage{enumitem}
\setlist[enumerate]{label=(\alph*),itemsep=5pt,topsep=8pt}
\setlist[itemize]{label=$\triangleright$,itemsep=5pt,topsep=6pt}

\Crefformat{enumi}{#2#1#3}

\usepackage{tikzit}				% String diagrams
\input{markov2.tex}

\usepackage{indentfirst}
\newcommand{\newsharedtheorem}[3]{%
	\newaliascnt{#1}{dummy}%
	\newtheorem{#1}[#1]{#2}%
	\aliascntresetthe{#1}%
	\crefname{#1}{#2}{#3}%
}
\newsharedtheorem{theorem}{Theorem}{Theorems}
\newsharedtheorem{lemma}{Lemma}{Lemmas}
\newsharedtheorem{proposition}{Proposition}{Propositions}
\newsharedtheorem{corollary}{Corollary}{Corollaries}
\newsharedtheorem{conjecture}{Conjecture}{Conjectures}
\newsharedtheorem{question}{Question}{Questions}
\newsharedtheorem{definition}{Definition}{Definitions}
\newsharedtheorem{problem}{Problem}{Problems}
\newsharedtheorem{slogan}{Slogan}{Slogans}
\theoremstyle{remark}
\newsharedtheorem{example}{Example}{Examples}
\newsharedtheorem{remark}{Remark}{Remarks}
\numberwithin{equation}{section}

\newcommand{\copi}{\mathsf{copy}}
\newcommand{\del}{\mathsf{del}}
\newcommand{\FinSet}{\mathsf{FinSet}}
\newcommand{\Set}{\mathsf{Set}}

\newcommand{\FinStoch}{\mathsf{FinStoch}}
\newcommand{\Stoch}{\mathsf{Stoch}}
\newcommand{\BorelStoch}{\mathsf{BorelStoch}}
\newcommand{\cC}{\mathsf{C}}
\newcommand{\cD}{\mathsf{D}}

\newcommand{\Hyp}{\mathsf{Hyp}}
\newcommand{\CatHyp}{\mathsf{CatHyp}}
\newcommand{\FinHyp}{\mathsf{FinHyp}}
\newcommand{\cI}{\mathsf{I}}
\newcommand{\Cat}{\mathsf{Cat}}
\newcommand{\MonCat}{\mathsf{MonCat}}
\newcommand{\gsCat}{\mathsf{gsCat}}
\newcommand{\MarkovCat}{\mathsf{MarkovCat}}

\newcommand{\cospan}{\mathsf{cospan}}
\newcommand{\FreeGS}{\mathsf{FreeGS}}
\newcommand{\FreeMarkov}{\mathsf{FreeMarkov}}
\newcommand{\Bloom}{\mathsf{Bloom}}
\newcommand{\op}{\mathrm{op}}

\newcommand{\id}{\mathrm{id}}
\newcommand{\N}{\mathbb{N}}
\newcommand{\Iin}{\mathsf{in}}
\newcommand{\Iout}{\mathsf{out}}

\newcommand{\hyp}{\mathsf{hyp}}		% underlying hypergraph
\newcommand{\norm}{\mathsf{norm}}	% normalization

\newcommand{\newterm}[1]{\textbf{#1}}

\title{Free gs-monoidal categories and free Markov categories}

\author{Tobias Fritz}
\address{Department of Mathematics, University of Innsbruck, Austria}
\email{tobias.fritz@uibk.ac.at}

\author{Wendong Liang}
\address{Department of Mathematics, École Normale Supérieure Paris-Saclay, France}
\email{wendong.liang@ens-paris-saclay.fr}


\date{\today}

\begin{document}

\begin{abstract}
	Categorical probability has recently seen significant advances through the formalism of Markov categories, within which several classical theorems have been proven in entirely abstract categorical terms.
	Closely related to Markov categories are gs-monoidal categories, also known as CD categories. 
	These omit a condition that implements the normalization of probability.
	Extending work of Corradini and Gadducci, we construct free gs-monoidal and free Markov categories generated by a collection of morphisms of arbitrary arity and coarity.
	For free gs-monoidal categories, this comes in the form of an explicit combinatorial description of their morphisms as structured cospans of labeled hypergraphs.
	These can be thought of as a formalization of gs-monoidal string diagrams ($=$term graphs) as a combinatorial data structure.
	We formulate the appropriate $2$-categorical universal property based on ideas of Walters and prove that our categories satisfy it.

	We expect our free categories to be relevant for computer implementations and we also argue that they can be used as statistical causal models generalizing Bayesian networks.
\end{abstract}

\maketitle

\tableofcontents

\newpage

\section{Introduction}

The categorical approach to probability theory has recently seen significant advances in the form of purely categorical proofs of a number of classical results in probability and theoretical statistics, such as the de Finetti theorem~\cite{fritz2020synthetic,fritzrischel2019zeroone,fritz2020representable,fritz2021definetti,jacobs2021hyper}.
This progress has been enabled by the advent of \newterm{Markov categories}.
These provide a categorical framework for probability that seems to capture the qualitative\footnote{In the sense of ``non-quantitative''.} intrinsic structure of probability remarkably accurately.
The usual Markov category in which probability theory is considered to take place is $\BorelStoch$, the category with standard Borel measurable spaces as objects and Markov kernels as morphisms.
But there are many other Markov categories~\cite[Sections~3--9]{fritz2020synthetic},  and one can instantiate the definitions and theorems of categorical probability in those as well, at least to the extent that they satisfy the relevant additional axioms.\footnote{For the most part, such instantiations have not been worked out yet.}

Closely related to Markov categories are \newterm{gs-monoidal categories}, which differ from Markov categories only by dropping an axiom which corresponds to the normalization of probability~\cite{cho2019disintegration}. In the strict case, these go back to Gadducci's thesis~\cite[Definition~3.9]{gadducci1996thesis} and a corresponding paper of Corradini and Gadducci~\cite{corradini1999gsmonoidal}, where they were considered with a different motivation in the context of term graphs and term graph rewriting.
Their intended application had suggested to study gs-monoidal categories freely generated by a given collection of morphisms, and indeed these had already been constructed and characterized in the single-sorted case with generators of coarity one~\cite{corradini1999gsmonoidal} and later defined in general~\cite{coccia2003lambda,bruni2014gs}.
Apparently unaware of the work of Corradini and Gadducci, free gs-monoidal categories made their first appearance in the probability context somewhat implicitly in the Master's thesis of Fong~\cite{fong2013causal}, with the aim of providing a categorical framework for Bayesian networks: a Bayesian network on any directed acyclic graph $G$ can be \emph{defined} as a gs-monoidal functor from a gs-monoidal category associated to $G$ to $\FinStoch$ if one has discrete variables, or to $\BorelStoch$ if one wants to allow continuous variables as well.
Finally, another independent appearance of free gs-monoidal categories is in the recent thesis of Stein~\cite[Section~7.3]{stein2021probabilistic}, where a construction as syntactical categories was given.

In this paper, we improve on these previous works in several ways:
\begin{itemize}
	\item We explicitly construct free gs-monoidal categories generated by morphisms of arbitrary arity and coarity and with arbitrarily many sorts (\Cref{freegs}).
	\item Our construction defines the morphisms as certain structured cospans, which makes the connection with the recent categorical approach to network theory explicit.
		These cospans can be thought of as a combinatorial description of gs-monoidal string diagrams, which underlines their network-like nature.
	\item We prove that our free gs-monoidal categories satisfy the ``right'' $2$-categorical universal property---rather than merely a $1$-categorical one---and without any strictness requirement on the target categories (\Cref{univ_prop}).
	\item We argue that morphisms in free gs-monoidal categories can themselves be considered as causal models generalizing the DAGs that underlie Bayesian networks (\Cref{causal}).
\end{itemize}
Moreover, in \Cref{bloom_circuitry} we show that every free gs-monoidal category comes equipped with a canonical factorization system that we call the \newterm{bloom-circuitry factorization}: every gs-monoidal string diagram factors uniquely into a ``pure bloom'', where all its boxes appear and every wire gets copied so as to connect to a unique output interface, followed by ``pure circuitry'' where wires get copied and discarded.
In \Cref{freemarkov}, we explain how our characterization can be adapted from free gs-monoidal categories to free Markov categories.
Finally, \Cref{causal} provides a sketch of how the morphisms in a free gs-monoidal category, or equivalently gs-monoidal string diagrams, can be viewed and utilized in statistics as causal models generalizing Bayesian networks.
This has been explored in more detail in the follow-up paper~\cite{fritz2023dseparation}.

\subsection*{Related work}

There has been independent and concurrent work by Milosavljevic and Zanasi~\cite{milosavljevic2022rewriting}, where free gs-monoidal categories have been constructed in the single-sorted case and a $1$-categorical universal property has been considered (essentially the same as our \Cref{cdterm_free1}); and by Yin~\cite{yin2022free}, where a topological description of free Markov categories has been given.

\medskip

Next, we discuss several points in a bit more detail.

\subsection*{Universal property and existence of free categories}

As with free categories with extra structure in general, it is not obvious what a suitable universal property of a free gs-monoidal category or free Markov category should be.
There are three meaningful choices, and we will prove that all of these properties hold. In increasing generality, these are:
\begin{itemize}
	\item A $1$-categorical universal property characterizing the free gs-monoidal/Markov category as a universal object in a category of colored PROPs using strict symmetric monoidal identity-on-objects functors as morphisms (second part of \Cref{cdterm_free1}).
		This is the type of universal property that has been considered for example in~\cite{enriquez2005functors,vallette2003props,fresse2010props} for free PROPs and in~\cite{bonchi2016rewriting,zanasi2017rewriting} for free hypergraph categories.

		Since colored PROPs with the relevant extra structure are the models of a multi-sorted algebraic theory, which is well-known to have free models, it is clear that free gs-monoidal categories and free Markov categories with this universal property exist: one simply needs to consider all possible expression formed out of the generators and quotient by the defining laws. For free gs-monoidal categories, this has been considered explicitly in~\cite[Definition~2.4]{coccia2003lambda} and~\cite[Definition~2]{bruni2014gs}.
		However, finding a concrete combinatorial description is a different and more involved problem.
	\item The same $1$-category universal property, generalized to strong symmetric monoidal functors that are not necessarily identity-on-objects (first part of \Cref{cdterm_free1}).
		This is the type of universal property considered by Corradini and Gadducci~\cite{corradini1999gsmonoidal} for strict gs-monoidal categories, and in~\cite{hackney2015category} for free colored PROPs.
	\item A full-fledged $2$-categorical universal property (\Cref{cdterm_free}).
		This property is somewhat more involved, but it is the one that we expect to be most relevant in practice, where identity-on-objects functors appear quite rarely.
		The formulation of this universal property follows an idea of Walters~\cite{walters1989free}.
		As far as we know, this type of universal property has not been considered much to date.
\end{itemize}
We place our emphasis on the third universal property, and merely note that the other two follow as a byproduct of essentially the same arguments.

In any case, it is clear that the mere existence of such free categories is not very useful: having a more explicit description of their hom-sets is a central tool for working with them in practice, and finding such a description is our main goal.

\subsection*{Are string diagrams topological or combinatorial?}

Since the morphisms in a free gs-monoidal category are gs-monoidal string diagrams,
this paper also belongs to a line of work in which string diagrams are viewed as morphisms in free monoidal categories (possibly with extra structure) and characterized as suitable combinatorial objects.
Among the earlier works approaching free symmetric monoidal categories in this way are results on free PROPs by Enriquez and Etingof~\cite[Lemma~2.1]{enriquez2005functors} and Vallette~\cite[p.~47]{vallette2003props}.\footnote{See also Fresse~\cite[Appendix~A]{fresse2010props} and Hackney and Robertson~\cite[Appendix~A]{hackney2015category}.}
This contrasts with the more traditional topological approach, in which string diagrams are considered topological entities, as in the original work of Joyal and Street~\cite{joyalstreet1991tensor}.
This distinction is analogous to the one in graph theory, where a graph can be considered (with somewhat different meaning) either as a continuous topological structure, for example as embedded on a surface, or as a purely combinatorial structure.

The topological formalization of string diagrams is adequate in particular in situations where higher categorical considerations or the cobordism hypothesis~\cite{baezdolan1995tqft} play a role.
In our present context, we believe that the combinatorial formalization is clearly preferable.
This sentiment has also been expressed in the recent work on string diagram rewriting by Bonchi et al.~\cite{bonchi2021rewrite}.
Indeed the combinatorial approach is both more useful in practice (for example for computer formalization) and more faithful to the intuition of a string diagram as a kind of ``network'' than a topological picture would be.
The fact that the relevant combinatorial structure is given by cospans of hypergraphs makes the combinatorial characterization match up nicely with the \newterm{structured cospan} approach to network theory~\cite{baezcourser2020structured}.\footnote{A minor technical difference is that our additional restrictions on the leg of a cospan, in the form of a monogamy condition adopted from~\cite{bonchi2016rewriting,zanasi2017rewriting}, have not been considered in the literature on categorical network theory so far. Thus our construction is not technically an instance of the structured cospan framework in its present form.}

However, it is conceivable that interesting topological constructions of free Markov categories can be given as well.
We refer to the independent and concurrent work of Yin for more on this~\cite{yin2022free}.

\subsection*{Applications}

One expected application is to computer implementations of gs-monoidal string diagrams.
As with the existing work on rewriting of string diagrams in symmetric monoidal and hypergraph categories~\cite{bonchi2016rewriting,zanasi2017rewriting,patterson2021wiring}, it is of interest for computer implementations to have a data structure for representing string diagrams in gs-monoidal categories and Markov categories, and moreover to have a rewriting theory for these.
While the former is given in this paper, the latter still remains to be developed.

Another expected area of application is to causal statistical models.
As we argue in \Cref{causal}, gs-monoidal string diagrams constitute a natural generalization of Bayesian networks which already incorporates features like latent variables or input variables without the need for additional annotations.
This is in line with recent work using gs-monoidal string diagrams in the context of causal inference~\cite{jacobs2021surgery,gitton2022inflation}\footnote{Although~\cite{gitton2022inflation} uses different terminology, its ``tensor networks'' are manifestly the same as gs-monoidal string diagrams.}.

\subsection*{Free Markov categories as examples}

Besides the possible importance of free Markov categories for the theory of causal models and causal inference, free gs-monoidal and free Markov categories constitute an addition to the bestiary of examples of such categories, and as such may also be useful as a testing ground for future conjectures in categorical probability.
In order to situate them in the landscape of examples, it would be useful to determine under what conditions they satisfy the additional axioms introduced in~\cite[Section~11]{fritz2020synthetic}.

While detailed definitions will be given in the main text, for now suffice it to say that the data that generates a free Markov category $\FreeMarkov_\Sigma$ is a collection of ``boxes'' $\Sigma$, called a \emph{monoidal signature}. These are exactly those boxes that appear in the string diagrams that form the morphisms of $\FreeMarkov_\Sigma$.

\begin{question}
	Under what conditions on $\Sigma$ does the free Markov category $\FreeMarkov_\Sigma$:
	\begin{enumerate}[label=(\roman*)]
		\item have conditionals?
		\item satisfy the causality axiom?
		\item satisfy the positivity axiom?
	\end{enumerate}
\end{question}

Also, the representability of a Markov category~\cite{fritz2020representable} is an important property relevant to categorical probability. 

\begin{conjecture}
	A free Markov category $\FreeMarkov_\Sigma$ is representable if and only if $\Sigma$ contains no boxes with at least one output.
\end{conjecture}

Adding or removing a box with no outputs to $\Sigma$ leaves $\FreeMarkov_\Sigma$ invariant. Hence the conjecture equivalently states that the free Markov categories of that form are exactly the symmetric monoidal categories freely generated by a commutative comonoid on every generating object.
These are well-known to be those of the form $(\FinSet / W)^\op$ for any set $W$ with the cartesian monoidal structure.

To approach the conjecture, one may want to prove first that the deterministic morphisms in a free Markov category are exactly the pure circuitry ones.
We leave the details to future work.

\subsection*{Notation and conventions}

All categories in this paper will be assumed locally small.
We use the term \emph{$2$-category} in the strict sense.
$\Cat$ is the usual $2$-category of categories, and similarly $\MonCat$ is the $2$-category of monoidal categories, strong monoidal functors and monoidal transformations.

\subsection*{Acknowledgments}

The first author acknowledges funding by the Austrian Science Fund (FWF) through project P 35992-N.
We thank Andreas Klingler for discussions and fruitful collaboration on~\cite{fritz2023dseparation}.
We also thank Peter Czaban, Arthur Parzygnat, Richard Samuelson, Rob Spekkens, Elie Wolfe and Fabio Zanasi for discussions and helpful feedback on a draft.
Special thanks are due to Fabio Gadducci for a long email exchange and copious help with the literature on gs-monoidal categories.

\emph{Update:} We also thank John Baez, Nathaniel Osgood and Claude for pointing out that our earlier version of \Cref{cdterm_free} was erroneous.

\section{GS-monoidal categories and Markov categories}

We put our main focus here on stating the main definitions and refer to~\cite{fritz2020synthetic,fritzrischel2019zeroone,fritz2020representable,fritz2021definetti,jacobs2021hyper,cho2019disintegration} for applications to probability and statistics.
We comment on the history in \Cref{history}.

\begin{definition}
	\label{cd_cat}
	A \newterm{gs-monoidal category} is a symmetric monoidal category $(\cC, \otimes, I)$ with a
	commutative comonoid structure on each object $X$, consisting of a comultiplication and counit,
	\begin{equation}
			\input{copy} \hspace{2cm}
			\input{del}
	\end{equation}
	which satisfy the commutative comonoid equations,
	\begin{equation}\begin{split}\label{comonoid_eq}
			\input{commutative_comonoid_eq}
	\end{split}\end{equation}
	These comonoid structures must be multiplicative with respect to the monoidal structure:
	\begin{align}\begin{split}\label{compatible_monoid}
			\input{compatible_monoidal1} \\
			\input{compatible_monoidal2}
	\end{split}\end{align}
	A gs-monoidal category is a \newterm{Markov category} if the monoidal unit $I$ is terminal.\footnote{Note that this makes the second set of equations in~\eqref{compatible_monoid} trivially true.}
\end{definition}

In this definition, ``gs'' stands for \emph{garbage} (deletion) and \emph{sharing} (copy), as per the intended interpretation of these comonoid structures.
By definition, a Markov category is equivalently a gs-monoidal category which satisfies the naturality of discarding,
\begin{equation}
	\begin{split}
	\label{nat_counit}
	\input{nat_counit}
	\end{split}
\end{equation}
for all morphisms $f$. This corresponds to the normalization of probability. 
We refer to~\cite[Section~3--9]{fritz2020synthetic} for a list of examples of Markov categories.

\begin{remark}[History]
	\label{history}
	The notions of gs-monoidal category and Markov category already have a multifarious history.
	As far as we know, gs-monoidal categories made their first appearance in the 1996 thesis of Gadducci~\cite[Definition~3.9]{gadducci1996thesis}, with the minor difference that they were required to be strict.
	This was motivated by graphical notations for formal languages, and is relevant more generally for resource sharing in computer science~\cite{hasegawa1997sharing}.
	In 1999, a similar definition was given in independent work of Golubtsov~\cite{golubtsov1999axiomatic} as \emph{categories of information transformers}, already with applications to statistics in mind.\footnote{There also are a number of overlapping follow-up papers by Golubtsov, see~\cite{fritz2020synthetic} further discussion and the references.}

	Subsequently, other works on string-diagrammatic approaches of probability appeared, such as by Coecke and Spekkens~\cite{coecke2012picturing}, apparently unaware of the above earlier works and using somewhat different categorical axiomatics.
	The first work in probability to consider something very close to gs-monoidal categories as defined above seems to have been Fong's 2013 Master's thesis~\cite{fong2013causal}, although no explicit definition of the relevant kind of category is given.
	GS-monoidal categories were considered in their present form as \emph{CD-categories} in a 2019 paper of Cho and Jacobs~\cite{cho2019disintegration}.
	Markov categories were introduced there as well under the name \emph{affine CD-categories}.
	The term \emph{Markov category} was then introduced in a 2020 paper by the first-named present author~\cite{fritz2020synthetic}, based on the interpretation of the morphisms as generalized Markov kernels.
\end{remark}

The following $2$-category of gs-monoidal categories does not seem to have been considered before.
However, it is an instance of the more general categorical machinery of~\cite[Definition~4.1]{fongspivak2019bells}, and its Markov version is~\cite[Definition~10.14]{fritz2020synthetic}.

\begin{definition}
	Let $\cC$ and $\cD$ be gs-monoidal categories.
	\label{cd_cells}
	\begin{enumerate}
		\item A \newterm{strong gs-monoidal functor} $F : \cC \to \cD$ is a symmetric monoidal functor such that the diagrams
			\[
				\begin{tikzcd}
					& F(X) \ar{dl}[swap]{F(\copi_X)} \ar{dr}{\copi_{F(X)}} \\
					F(X \otimes X) \ar{rr}{\cong}	& & F(X) \otimes F(X)
				\end{tikzcd}
				\qquad\quad
				\begin{tikzcd}
					& F(X) \ar{dl}[swap]{F(\del_X)} \ar{dr}{\del_{F(X)}} \\
					F(I) \ar{rr}{\cong}	& & I
				\end{tikzcd}
			\]
			commute for all $X \in \cC$, where the horizontal isomorphisms are the coherence isomorphisms of $F$.
		\item A \newterm{gs-monoidal transformation} is a monoidal natural transformation between gs-monoidal functors.
	\end{enumerate}
	We write $\gsCat$ for the $2$-category of gs-monoidal categories, strong gs-monoidal functors and gs-monoidal transformations.
\end{definition}

As with monoidal functors in general, the definition of strong gs-monoidal functor has a couple of variants: a \newterm{lax gs-monoidal functor} is a symmetric lax monoidal functor such that the above triangles commute with the laxators in place of the coherence isomorphisms. An oplax gs-monoidal functor is defined analogously. A \newterm{strict gs-monoidal functor} is a symmetric monoidal functor $F : \cC \to \cD$ which preserves the gs-monoidal structure on the nose,
\[
	F(\copi_X) = \copi_{FX}, \qquad F(\del_X) = \del_{FX}.
\]
If $\cC$ and $\cD$ are Markov categories, then we also speak of strong/lax/strict \newterm{Markov functors} as well as \newterm{Markov transformations}, resulting in a $2$-category of Markov categories that we denote by $\MarkovCat$.
It is a full sub-$2$-category of $\gsCat$ by definition.

\section{Construction of free gs-monoidal categories}
\label{freegs}

We start with some preparation by introducing the relevant combinatorial structures, following the definitions used by Bonchi et al.~\cite{bonchi2016rewriting} with some minor adaptations.

\begin{definition}
	\label{hypergraph}
	We write $\cI$ for the category where:
	\begin{itemize}
		\item Objects are pairs of natural numbers\footnote{Our convention is $0 \in \N$.} $(k,\ell) \in \mathbb{N} \times \mathbb{N}$, and there is one extra object denoted $*$.
		\item There are $k + \ell$ non-identity morphisms from $(k,\ell)$ to $*$, denoted
			\[
				\Iin_1, \: \ldots, \: \Iin_k, \: \Iout_1, \: \ldots, \: \Iout_\ell \; : \; (k, \ell) \longrightarrow *,
			\]
			and no other non-identity morphisms.
	\end{itemize}
	A \newterm{hypergraph} is a functor $G : \cI \to \Set$.\footnote{A more unambiguous term would be \newterm{directed hypergraph}, but we omit the adjective for brevity.}$^,$\footnote{Note that the definition used in \cite{bonchi2016rewriting} and~\cite{zanasi2017rewriting} requires the functor to land in $\FinSet$. We prefer to allow infinite hypergraphs, since otherwise we could only consider free gs-monoidal categories generated by finitely many morphisms, which would be unnecessarily restrictive.}
\end{definition}

A few remarks may help clarify the definition.

\begin{itemize}
	\item We do not need to define composition in $\cI$ since there are no composable pairs of non-identity morphisms.
	\item A hypergraph $G : \cI \to \FinSet$ has a set of nodes given by $G(*)$. We will call nodes \newterm{wires} because of the role they play for string diagrams. Moreover, for all $k,\ell \in \N$ there is a set $G(k,\ell)$ of hyperedges, which we will call \newterm{boxes}, with $k$ input wires and $\ell$ output wires. The maps
		\[
			G(\Iin_1), \: \ldots, \: G(\Iin_k), \: G(\Iout_1), \ldots, \: G(\Iout_\ell) \; : \; G(k, \ell) \longrightarrow G(*)
		\]
		assign to every such box its lists of input wires and output wires.
		In particular, the input and output wires of a box are both ordered.
	\item In order to have a more intuitive notation, we also write
		\[
			W(G) \coloneqq G(*), \qquad\quad B_{k,\ell}(G) \coloneqq G(k,\ell), \qquad\quad B(G) \coloneqq \coprod_{k,\ell \in \N} B_{k,\ell}(G)
		\]
		for the sets of wires and boxes.
		We also abbreviate $G(\Iin_i)$ and $G(\Iout_j)$ to $\Iin_i$ and $\Iout_j$, respectively.
		For a box $b$ and wire $w$, we put
		\begin{align*}
			\Iin(b,w) & \coloneqq |\{ i = 1, \ldots, k \: \mid \: \Iin_i(b) = w \}|, \\
			\Iout(b,w) & \coloneqq |\{ j = 1, \ldots, \ell \: \mid \: \Iout_j(b) = w \}|,
		\end{align*}
		for the number of times that $w$ appears as an input or output wire of $b$, respectively.
	\item In~\cite{hackney2015category}, a very similar definition has been used under the term \emph{megagraph}. A megagraph comes equipped with additional actions of symmetric groups corresponding to permutations of inputs and outputs.
\end{itemize}

\begin{definition}
	For hypergraphs $G$ and $H$, a \newterm{hypergraph morphism} $G \to H$ is a natural transformation $G \Rightarrow H$.
\end{definition}

\begin{itemize}
	\item A morphism of hypergraphs $p : G \Rightarrow H$ thus consists of a map $p(*) : W(G) \to W(H)$, taking wires to wires, and maps $p(k,\ell) : G(k, \ell) \to H(k, \ell)$, taking boxes to boxes, such that the diagrams
		\[\begin{tikzcd}
			{B_{k,\ell}(G)} && {B_{k,\ell}(H)} && {B_{k,\ell}(G)} && {B_{k,\ell}(H)} \\ \\
			W(G) && W(H) && W(G) && W(H)
			\arrow["\Iin_i"', from=1-1, to=3-1]
			\arrow["\Iin_i", from=1-3, to=3-3]
			\arrow["p_{k,\ell}"', from=1-1, to=1-3]
			\arrow["p_*", from=3-1, to=3-3]
			\arrow["\Iout_j"', from=1-5, to=3-5]
			\arrow["\Iout_j", from=1-7, to=3-7]
			\arrow["p_{k,\ell}"', from=1-5, to=1-7]
			\arrow["p_*", from=3-5, to=3-7]
		\end{tikzcd}\]
		commute for all $i=1,\ldots,k$ and $j=1,\ldots,\ell$.
		These conditions state exactly that the input and output wires of each box must be preserved.
	\item We thereby obtain a category of hypergraphs denoted $\Hyp$, which by definition is exactly the presheaf topos $\Set^\cI$.
\end{itemize}

\begin{remark}
	\label{internal_hyp}
	Every monoidal category $\cC$ has an underlying hypergraph, in which the wires are the objects of $\cC$ and the boxes of arity $(k,\ell)$ are the morphisms from a $k$-fold tensor product of objects to an $\ell$-fold tensor product of objects.

	Following an idea of Walters~\cite{walters1989free}, we can refine this construction to a $2$-functor
	\[
		\hyp \: : \: \MonCat \longrightarrow \CatHyp,
	\]
	where $\CatHyp$ is the $2$-category of categories internal to $\Hyp$.
	Indeed if $\cC$ is a small monoidal category, then it has an underlying $\hyp(\cC) \in \CatHyp$ defined as follows.
	Its category of wires is exactly $\cC$,
	\[
		W(\hyp(\cC)) \coloneqq \cC,
	\]
	while the category of boxes $B_{k,\ell}(\hyp(\cC))$ is the ``hyperarrow category'' of $\cC$, by which we mean the category with set of objects given by
	\[
		\coprod_{A_1, \ldots, A_k, B_1, \ldots, B_\ell \,\in\, \cC} \cC(A_1 \otimes \cdots \otimes A_k, B_1 \otimes \cdots \otimes B_\ell),
	\]
	and where a morphism from an object
	\[
		f \: : \: A_1 \otimes \cdots \otimes A_k \longrightarrow B_1 \otimes \cdots \otimes B_\ell
	\]
	to another object
	\[
		g : A'_1 \otimes \cdots \otimes A'_k \longrightarrow B'_1 \otimes \cdots \otimes B'_\ell
	\]
	consists of families of morphisms $A_i \to A'_i$ and $B_j \to B'_j$ making the obvious square commute, and with composition defined in the obvious way.
	The source and target functors $B_{k,\ell}(\hyp(\cC)) \to W(\hyp(\cC))$ are also the obvious ones.

	If $F : \cC \to \cD$ is a strong monoidal functor, then $F$ induces an internal functor
	\[
		\hyp(F) \: : \:	\hyp(\cC) \longrightarrow \hyp(\cD)
	\]
	defined as just $F$ on the category of wires.
	On the category of boxes, it is defined on objects in terms of the composite
	\[
		\begin{tikzcd}
			\cC(A_1 \otimes \cdots \otimes A_k, B_1 \otimes \cdots \otimes B_\ell) \ar{r}{F} & \cD(F(A_1 \otimes \cdots \otimes A_k), F(B_1, \otimes \cdots \otimes B_\ell)) \ar{d} \\
			& \cD(F(A_1) \otimes \cdots \otimes F(A_k), F(B_1) \otimes \cdots \otimes F(B_\ell)),
		\end{tikzcd}
	\]
	where the second arrow is given by composition with the relevant coherence isomorphisms of $F$.
	On morphisms in the category of boxes, $\hyp(F)$ acts as just $F$ componentwise on morphisms.
	It is straightforward to see that this satisfies functoriality.

	For strong monoidal functors $F,\, G : \cC \to \cD$, a monoidal natural transformation $\alpha : F \Rightarrow G$ induces an internal natural transformation
	\[
		\hyp(\alpha) \: : \: \hyp(F) \Longrightarrow \hyp(G)
	\]
	defined as $\alpha$ at the level of wires.
	To a box object $f \in B_{k,\ell}(\hyp(\cC))$, it assigns the family of morphisms
	\[
		\alpha_{\Iin_1(b)}, \: \ldots, \: \alpha_{\Iin_k(b)}, \: 
		\alpha_{\Iout_1(b)}, \: \ldots, \: \alpha_{\Iout_\ell(b)},
	\]
	which indeed has the correct type, and where naturality is a straightforward consequence of the naturality of $\alpha$.

	It is straightforward to see that the above definitions assemble to a $2$-functor $\hyp : \MonCat \to \CatHyp$ as promised above.
	Furthermore, by restriction along the forgetful $2$-functor $\gsCat \to \MonCat$, we obtain a $2$-functor $\hyp : \gsCat \to \CatHyp$ as well, which is how we will understand $\hyp$ from now on.
\end{remark}

Hypergraphs serve a dual purpose in our context: first, the data that generates a free gs-monoidal category is a hypergraph $\Sigma$ that we call the \newterm{monoidal signature}.\footnote{Also called \emph{hypersignature} in~\cite{coccia2003lambda}.}
It may have infinitely many wires and/or boxes.  See \Cref{monoidal_sig} for an example. Second, the \emph{morphisms} in a free gs-monoidal category are themselves constructed in terms of finite hypergraphs.
This is based on the idea that every string diagram is a hypergraph with wires and boxes, with given sequences of input and output wires for each box~\cite{bonchi2016rewriting}.
For a string diagram in a gs-monoidal category, the input and output wires of a copy morphism are considered as one and the same wire, so that the ``wires'' in our sense are actually the \emph{connected components} of the ``circuitry'' built out of copies and discards. \Cref{cd_string_hypergraph} provides an example that we discuss in more detail below.

The formal construction that follows now elaborates on this intuitive idea: we will \emph{define} a gs-monoidal string diagram as a hypergraph with additional labeling and interfacing information.

\begin{definition}
	\label{hypergraph_finite}
	A hypergraph $G$ is:
	\begin{enumerate}[label=(\roman*)]
		\item \newterm{finite} if both the set of wires $W(G)$ and the set of boxes $B(G)$ are finite.
		\item \newterm{discrete} if there are no boxes, $B(G) = \emptyset$.
	\end{enumerate}
\end{definition}

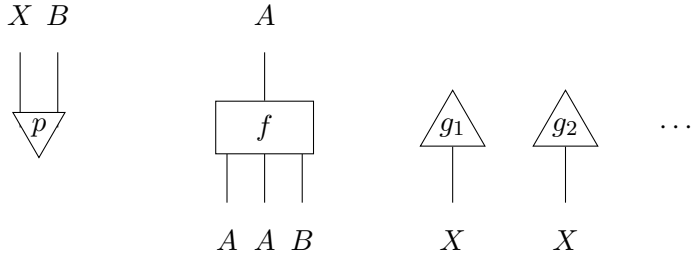
\begin{figure}
	\input{monoidal_signature}
	\caption{An example of a hypergraph $\Sigma$, drawn as a monoidal signature, with set of wires $W(\Sigma) = \{A,B,X\}$ and boxes $B(\Sigma) = \{p,f,g_1,g_2,\ldots\}$. Note that the input and output wires of a box are ordered, while the sets $W(\Sigma)$ and $B(\Sigma)$ themselves are not.}
	\label{monoidal_sig}
\end{figure}

Note that this finiteness condition is stronger than merely requiring $G$ to be a functor $\cI \to \FinSet$.\footnote{This has apparently has been missed in~\cite{bonchi2016rewriting}, where the functor $\cI \to \FinSet$ definition is stated but \Cref{hypergraph_finite} is used.}
The finite hypergraphs form a full subcategory of $\Hyp$ that we denote by $\FinHyp$.

String diagrams in gs-monoidal categories are ``directed'' in the sense that it is not possible to walk in a closed loop by traversing boxes from a source wire to a target wire.
This is formalized as follows.

\begin{definition}
	\label{acyclic}
	A \newterm{cycle} in a hypergraph $G$ is a finite alternating sequence of wires and boxes $(w_1,b_1,\ldots,w_n,b_n,w_{n+1})$ such that $w_{n+1} = w_1$ and
	\[
		\Iin(b_i, w_i) \ge 1, \qquad \Iout(b_i, w_{i+1}) \ge 1 
	\]
	for all $i = 1, \ldots, n$.
	We call $G$ \newterm{acyclic} if it does not have a cycle.
\end{definition}

\begin{figure}
	\input{cd_string_hypergraph}
	\caption{An example of a gs-monoidal categorical string diagram in the monoidal signature $\Sigma$ of \Cref{monoidal_sig}. The particular drawing in the plane is not part of the structure of the diagram, but the total orderings on input and output wires of each box as well as on input and output interfaces are.}
	\label{cd_string_hypergraph}
\end{figure}
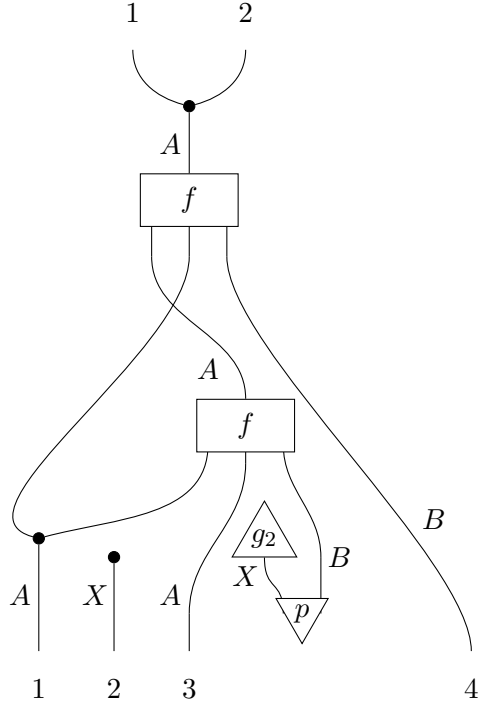

To construct the free gs-monoidal category generated by a monoidal signature $\Sigma$, we consider the slice category $\Hyp / \Sigma$, or rather its full subcategory $\FinHyp / \Sigma$ (though $\Sigma$ itself does not need to be finite). Its objects are hypergraph morphisms $G \to \Sigma$ for finite $G$, where we think of as such a morphism as a labelling of wires of $G$ by wires of $\Sigma$ and of boxes of $G$ by by boxes of $\Sigma$ of the same arity and coarity.
Morphisms are morphisms of hypergraphs that preserve the labeling.

In the following, we also write $n$ as shorthand for the set $\{1,\ldots,n\}$, and use underline notation $\underline{n}$ for the discrete hypergraph with set of wires $n$.

\begin{definition}
	\label{cd_diagram}
	Let $\Sigma$ be a hypergraph. Then a \newterm{gs-monoidal string diagram} for the monoidal signature $\Sigma$ is a cospan in $\FinHyp / \Sigma$ of the form
	\[\begin{tikzcd}
		& G \\
		\underline{m} \ar{ur}{p} & & \underline{n} \ar{ul}[swap]{q}
	\end{tikzcd}\]
	for an acyclic finite hypergraph $G$ and $m, n \in \N$, such that \newterm{left monogamy} holds: for every wire $w \in W(G)$, we have
	\[
		|p^{-1}(w)| + \sum_{b \in B(G)} \Iout(b,w) = 1.
	\]
	We also keep mention of the labeling morphisms $\sigma : G \to \Sigma$ implicit, which uniquely restrict to labeling morphisms $\sigma : \underline{m}, \underline{n} \to \Sigma$.
\end{definition}

The left leg of such a left monogamous cospan is a monomorphism.
The idea of left monogamy is borrowed from~\cite{bonchi2016rewriting}, where essentially the same condition was used for both legs as part of a combinatorial formalization of string diagrams in hypergraph categories.

The main point is now that a string diagram in a gs-monoidal category---or rather an equivalence class of string diagrams modulo the gs-monoidal category axioms---is precisely a gs-monoidal string diagram in the above sense.
While this idea is formalized by the universal property that we will prove in the next section, we illustrate it here intuitively with the example of \Cref{cd_string_hypergraph}.
There are $8$ connected components of wiring, which indicates that the hypergraph $G$ has $8$ wires, equipped with labels in $W(\Sigma)$.
There are $5$ boxes, indicating that $G$ has $5$ boxes, equipped with labels in $B(\Sigma)$.
The incidences between the wires and the boxes depict the hypergraph structure; the fact that the labeling $G \to \Sigma$ must be a hypergraph morphism ensures that the input and output wires of each box match the ones required by $\Sigma$ both in number and in their sequences of labels.
The acyclicity holds since all wires are consistently directed upward, although $\Sigma$ itself (\Cref{monoidal_sig}) need not be acyclic.
The \newterm{input interfaces} $1$ to $4$ at the bottom describe the left leg of the cospan: each interface connects to a wire in $G$ as determined by $p$.
Similarly, $q$ specifies how the \newterm{output interfaces} are connected to wires.
The left monogamy condition amounts to the fact that every component of wiring is either an output of a box (and in a unique way), or connects with a unique input interface, but not both.
\Cref{noncd_string_hypergraph} displays examples of cospans which violate the acyclicity or the left monogamy condition, respectively; these can still be interpreted as string diagrams in the free hypergraph category generated by $\Sigma$~\cite{zanasi2017rewriting}, but not as string diagrams in a gs-monoidal category.

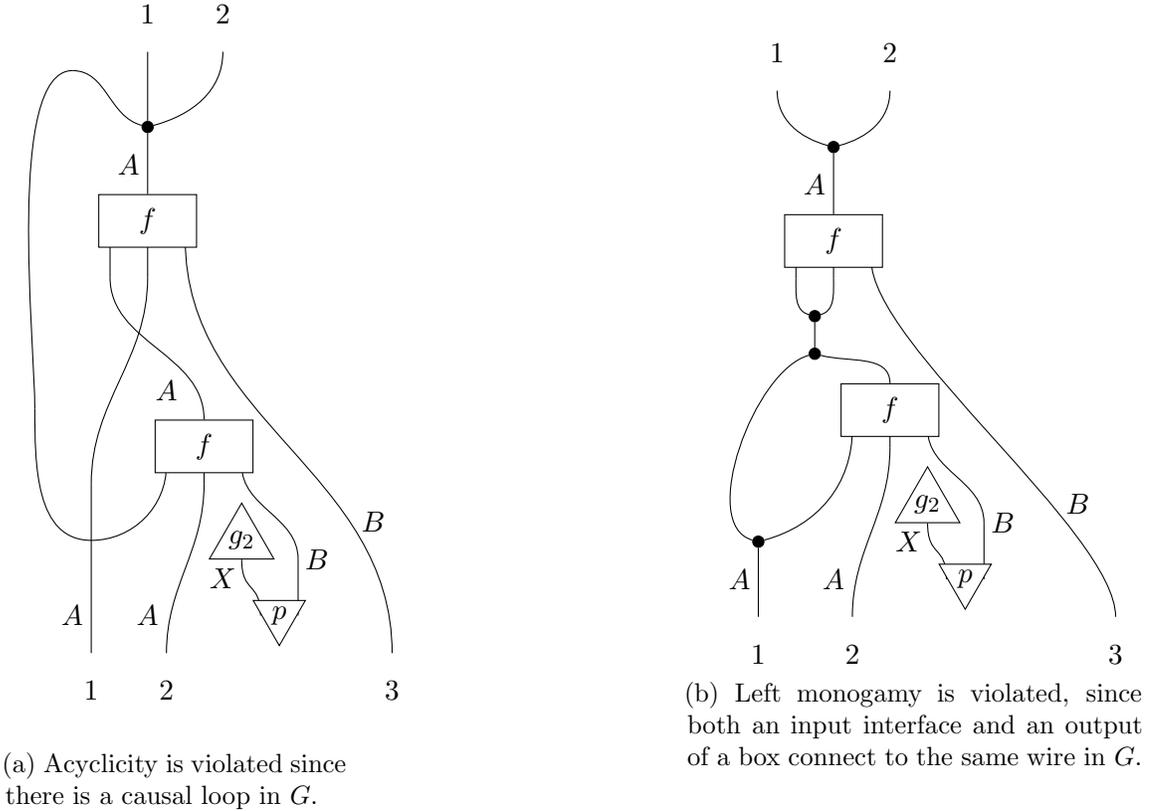
\begin{figure}
	\centering
	\begin{subfigure}{0.3\textwidth}
		\centering
		\input{noncd_string_hypergraph1}
		\caption{Acyclicity is violated since there is a causal loop in $G$.}
	\end{subfigure}
	\hfill
	\begin{subfigure}{0.4\textwidth}
		\centering
		\input{noncd_string_hypergraph2}
		\caption{Left monogamy is violated, since both an input interface and an output of a box connect to the same wire in $G$.}
	\end{subfigure}
	\caption{Examples of cospans of hypergraphs over the $\Sigma$ of \Cref{monoidal_sig} where either of the two combinatorial conditions of \Cref{cd_diagram} is violated. They can still be interpreted as string diagrams in hypergraph categories~\cite{zanasi2017rewriting}, which is how we draw them.}
	\label{noncd_string_hypergraph}
\end{figure}

The fact that $\FinHyp$ is a full subcategory of the presheaf topos $\Hyp$ closed under finite limits and colimits shows that pushouts in $\FinHyp$ exist and that they can be computed as separate pushouts in $\Set$ on wires and boxes of every arity $(k,\ell)$.
Since the forgetful functor $\FinHyp / \Sigma \to \FinHyp$ creates (finite) colimits, the same statement applies to the slice category $\FinHyp / \Sigma$.

It follows that we can form the category $\cospan(\FinHyp / \Sigma)$, where morphisms are isomorphism classes of cospans and composition is by pushout. This pushout is computed on wires and boxes of every arity separately, and the labels in $\Sigma$ carry over automatically.

\begin{definition}
	\label{cat_cd_diagrams}
	Given a hypergraph $\Sigma$, the \newterm{category of gs-monoidal string diagrams} $\FreeGS_\Sigma$ is the subcategory of $\cospan(\FinHyp / \Sigma)$ where:
	\begin{itemize}
		\item Objects are the pairs $(n, \sigma : \underline{n} \to \Sigma)$ with $n \in \N$.
		\item Morphisms are the isomorphism classes of gs-monoidal string diagrams (\Cref{cd_diagram}).
	\end{itemize}
\end{definition}

\begin{remark}
	\label{termgraphs_cat}
	In the special case of a single sort, meaning that $|W(\Sigma)| = 1$, and when all boxes in $\Sigma$ have exactly one output wire, an equivalent category has also been constructed in~\cite{corradini1999gsmonoidal} as a category of \emph{term graphs}.
\end{remark}

By definition, the objects $\underline{n} \to \Sigma$ of $\FreeGS_\Sigma$ are in bijection with maps $n \to W(\Sigma)$, or equivalently with words over $W(\Sigma)$, as expected for the free gs-monoidal category: it has a monoid of objects given by the free monoid with generating set $W(\Sigma)$.

Thus the composition of two gs-monoidal string diagrams is defined up to isomorphism, and is represented by the gs-monoidal string diagram obtained by pushout composition.
Suppressing the (uniquely induced) labeling $\sigma$ as usual, this pushout composition amounts to the composite cospan
\begin{equation}
	\label{pushout_comp}
	\begin{tikzcd}
		&& {G+_{\underline{m}} H} \\
		& G && H \\
		\underline{\ell} && \underline{m} && \underline{n}
		\arrow["p", from=3-1, to=2-2]
		\arrow["q"', from=3-3, to=2-2]
		\arrow["r", from=3-3, to=2-4]
		\arrow["s"', from=3-5, to=2-4]
		\arrow["t", from=2-2, to=1-3]
		\arrow["u"', from=2-4, to=1-3]
	\end{tikzcd}
\end{equation}
In order for \Cref{cat_cd_diagrams} to make sense, we need to prove that this composite is indeed a gs-monoidal string diagram again.
This is guaranteed by the following result, which has been obtained independently as~\cite[Lemma~22 and Proposition~28]{milosavljevic2022rewriting}\footnote{Under the additional assumption of a single sort ($|W(\Sigma)| = 1$), which however is does not simplify the relevant arguments.}.

\begin{lemma}
	If both constituent cospans in \eqref{pushout_comp} are gs-monoidal string diagrams, then so is the composite.
\end{lemma}

Independently of the proof, it is worth noting that $B(G +_{\underline{m}} H)$ is the disjoint union of $B(G)$ and $B(H)$ since $B(\underline{m}) = \emptyset$.
Moreover, $W(G +_{\underline{m}} H)$ is the disjoint union of $W(G)$ and $W(H)$, with every wire in $H$ that is in the image of $r$ identified with its (unique) counterpart wire in $G$.
These statements reflect obvious properties of string diagram drawings.

\begin{proof}
	We show first that the composite satisfies left monogamy.
	To start, the morphism $t$ is a pushout of the monomorphism $r$ and therefore a monomorphism as well\footnote{This holds for example because $\Hyp / \Sigma$ is a topos.}.
	In particular, the composite leg $tp$ is a monomorphism too.
	Given a wire $w \in W(G +_{\underline{m}} H)$, we prove that the left monogamy condition holds for $w$ by a case distinction:
	\begin{itemize}
		\item $w = t(w')$ for some $w' \in W(G)$.

			Then either $w'$ is a target of a box in $G$ in a unique way, or $w'$ is in the image of $p$, but not both.
			The left monogamy at $w$ follows in either case if we can show that $w$ is not a target of any box in the image of $u$.
			But this is indeed the case, since otherwise $w = u(\hat{w})$ for some $\hat{w} \in W(H)$.
			Since this makes $w$ a member of both the images of $t$ and $u$, it must actually come from $\underline{m}$.
			But then $\hat{w}$ is both an output of a box in $H$ and in the image of $r$, which contradicts the left monogamy of the second cospan.
		\item $w$ is not in the image of $t$.

			In this case, we must have $w = u(\hat{w})$ for some $\hat{w} \in W(H)$.
			The preimage $\hat{w}$ is unique and not in the image of $r$, since otherwise $w$ would also be in the image of $t$.
			Thus by left monogamy of the second cospan, $\hat{w}$ is a target of a box in $H$ in a unique way,
			and it follows that $w$ is a target of a box in the image of $u$ in a unique way.
	\end{itemize}

	Next, we show the acyclicity of the composite cospan.
	We prove that the existence of a cycle as in \Cref{acyclic} is impossible by another case distinction.
	\begin{itemize}
		\item If every $b_i$ is in the image of $t$, then this also applies to every $w_i$.
			This contradicts acyclicity of $G$, since taking preimages under $G$ produces a cycle there based on the fact that $t$ is a monomorphism.
		\item If no $b_i$ is in the image of $t$, then the boxes are all in the image of $u$. By left monogamy of the second cospan, all of the target wires of the corresponding boxes in $H$ are not in the image of $r$, and therefore they are the unique preimages under $u$ of all the target wires of the $b_i$ in $G +_{\underline{m}} H$.
			Hence we still obtain a cycle in $H$, contradicting the assumed acyclicity.
		\item If some $b_i$ are in the image of $t$ and some in the image of $u$, then there must be $i$ such that $b_i \in B(H)$ and $b_{i+1} \in B(G)$. But this not possible either, since by left monogamy of the second cospan, no target wire of $b_i$ comes from $\underline{m}$, which would be the only way in which it could also be an input wire of $b_{i+1}$.
			\qedhere
	\end{itemize}
\end{proof}

We have therefore defined the category $\FreeGS_\Sigma$, and we now explain how it becomes a gs-monoidal category.
To define the monoidal structure, note first that every $\underline{m + n}$ is a coproduct of $\underline{m}$ and $\underline{n}$ in $\FinHyp / \Sigma$.
Denoting the resulting copairing operation by angular brackets $\langle \cdot, \cdot\rangle$, the monoidal structure on objects is given by
\[
	(m, \sigma_1 : \underline{m} \to \Sigma) \otimes (n, \sigma_2 : \underline{n} \to \Sigma) \:\coloneqq\:
	(m + n, \langle \sigma_1, \sigma_2 \rangle : \underline{m + n} \to \Sigma),
\]
which under the canonical bijection between objects and words over $W(\Sigma)$ simply corresponds to the concatenation of words.

The monoidal structure on morphisms is given by
\begin{align}\begin{split}
	\label{parallel_comp}
	\left(
		\begin{tikzcd}[ampersand replacement=\&]
			\& G \\
			\underline{k} \ar{ur}{p} \& \& \underline{\ell} \ar{ul}[swap]{q}
		\end{tikzcd}
	\right) \otimes & \left(
		\begin{tikzcd}[ampersand replacement=\&]
			\& H \\
			\underline{m} \ar{ur}{r} \& \& \underline{n} \ar{ul}[swap]{s}
		\end{tikzcd}
	\right) \\[4pt]
	& \coloneqq \; \left(
		\begin{tikzcd}[ampersand replacement=\&]
			\& G + H \\
			\underline{k + m} \ar{ur}{p + r} \& \& \underline{\ell + n} \ar{ul}[swap]{q + s}
		\end{tikzcd}
	\right)
\end{split}\end{align}
again leaving it understood that the labels in $\Sigma$ are induced in the obvious way.
It is clear that the isomorphism class of the resulting cospan on the right only depends on the isomorphism classes of the factors, and its left monogamy and acyclicity follow directly from the definitions.
Moreover, the bifunctoriality of this tensor product follows straightforwardly from the fact that a coproduct of two pushout squares is a pushout square as well.
Since the tensor product is clearly associative, we therefore have equipped $\FreeGS_\Sigma$ with the structure of a strict monoidal category, where the monoidal unit is $\underline{0}$.

With symmetry isomorphisms given by
\[
	\begin{tikzcd}
		& \underline{m + n} \\
		\underline{m + n} \ar{ur}{\cong} & & \underline{n + m} \ar{ul}[swap]{\cong}
	\end{tikzcd}
\]
where the two legs are the obvious isomorphisms and the labeling is again left implicit, it is straightforward to see that $\FreeGS_\Sigma$ is strict symmetric monoidal.

To finally make $\FreeGS_\Sigma$ into a gs-monoidal category, it remains to equip every object $(n, f : \underline{n} \to \Sigma)$ with the structure of a comutative comonoid satisfying the relevant multiplicativity properties.
Again leaving the labeling $f$ implicit, the copy and discard maps are represented by the cospans
\begin{equation}
	\label{freecd_cd}
	\begin{tikzcd}
		& \underline{n} \\
		\underline{n} \ar{ur}{\id} & & \underline{n + n} \ar{ul}[swap]{\langle \id, \id \rangle}
	\end{tikzcd}
	\qquad \qquad
	\begin{tikzcd}
		& \underline{n} \\
		\underline{n} \ar{ur}{\id} & & \underline{0} \ar{ul}[swap]{\exists !}
	\end{tikzcd}
\end{equation}
We leave it to the reader to verify that the properties required of a gs-monoidal category indeed hold.

\begin{example}
	\label{degenerate}
	If the monoidal signature $\Sigma$ has no boxes, then a gs-monoidal string diagram degenerates to a cospan of finite sets in which the left leg is a bijection (by left monogamy), and therefore without loss of generality an identity morphism.
	In this way, $\FreeGS_\Sigma$ degenerates to a gs-monoidal category equivalent to $(\FinSet / W(\Sigma))^\op$, which is cartesian monoidal.
	This matches up with the intuition that string diagrams in this case consist of pure wirings, and the pure wirings are in bijection with all the ways of mapping output interfaces to input interfaces such that the types are preserved.
\end{example}

\begin{remark}
	\label{unit}
	For every generating box $b \in B_{k,\ell}(\Sigma)$, there is a hypergraph $\overline{b}$ containing just $b$, by which we mean the hypergraph with
	\[
		B(\overline{b}) = \{b\}, \qquad W(\overline{b}) = \{\Iin_1, \ldots, \Iin_k, \Iout_1, \ldots, \Iout_\ell\}.
	\]
	This hypergraph fits into a cospan
	\[
		\langle b \rangle \coloneqq
		\begin{tikzcd}
			& \overline{b} \\
			\underline{k} \ar{ur} & & \underline{\ell} \ar{ul}
		\end{tikzcd}
	\]
	where the two legs are simply the inclusions of the inputs and outputs of $b$ in their original order, respectively.
	When labeled in the obvious way, this cospan is the combinatorial representation of $b$ drawn as a string diagram as in \Cref{monoidal_sig}.

	In this way, we obtain a canonical hypergraph morphism $\Sigma \to \hyp(\FreeGS_\Sigma)$ mapping every generating wire to itself and every box $b$ to the associated gs-monoidal string diagram $\langle b \rangle$.
\end{remark}

\section{The universal property}
\label{univ_prop}

In this section, we make precise and then prove the claim that $\FreeGS_\Sigma$ is the free gs-monoidal category generated by $\Sigma$.
A similar result was proven by Corradini and Gadducci~\cite[Theorem~23]{corradini1999gsmonoidal}.
Ours improves on theirs in several ways:

\begin{itemize}
	\item We state and prove the ``correct'' $2$-categorical universal property in addition to the $1$-categorical one.
		This closely follows the ideas of Walters~\cite{walters1989free}, who had developed a $2$-categorical universal property for free categories with finite products and anticipated that this would apply very similarly to other categories with extra structure.
	\item Our result applies for monoidal signatures with arbitrary (finite) arity and coarity and with arbitrarily many sorts. This possibility had already been anticipated in~\cite[Section~6]{corradini1999gsmonoidal}.
	\item We do not assume that the target gs-monoidal categories are strict. Although assuming strictness is possible without loss of generality by strictification, it should be noted that the standard strictification theorem for symmetric monoidal categories is not sufficient, since it does not account for the gs-monoidal structure. Fortunately, a more general strictification theorem for monoidal categories with extra structure is available and applicable~\cite[Corollary~4.10]{fongspivak2019bells}.
		This indeed implies that every gs-monoidal category is equivalent to a strict one, in the sense of equivalence in the $2$-category $\gsCat$.
	\item There is a subtle gap in the proof~\cite{corradini1999gsmonoidal}. In our notation, the unproven statement is the assertion that if $|W(\Sigma)| = 1$, then the canonical functor $\FinSet \to \FreeGS_\Sigma$, which throws in the generating boxes in addition to the generating wire, is faithful~\cite[top of p.~315]{corradini1999gsmonoidal}. This is not obvious: there are examples of multi-sorted algebraic theories for which freely adding extra generators to a free model collapses the model~\cite{rezk2017multisorted}.
\end{itemize}

To begin, recall that in \Cref{internal_hyp} we associated to every gs-monoidal category $\cC$ a category $\hyp(\cC)$ internal to $\Hyp$, and this construction is a $2$-functor $\hyp : \gsCat \to \CatHyp$.
We now also consider a monoidal signature $\Sigma$ as a category internal to $\Hyp$, namely with the hypergraph of morphisms being $\Sigma$ itself, so that all morphisms are identities. Like this, the canonical hypergraph morphism $\Sigma \to \hyp(\FreeGS_\Sigma)$ from \Cref{unit} is an internal functor with the trivial action on morphisms.

To state and then prove the universal property,
we write
\[
	\CatHyp(\Sigma, \hyp(\cC))_{\det} \:\subseteq\: \CatHyp(\Sigma, \hyp(\cC))
\]
for the wide subcategory consisting of all internal functors and internal natural transformations whose components are deterministic morphisms in $\cC$,
meaning those that commute with the copy and discard morphisms in $\cC$.
Informally, the objects $\Sigma \to \hyp(\cC)$ are instances of the monoidal signature $\Sigma$ in $\cC$,
while a morphism equips each generating wire with a deterministic morphism in $\cC$, such that these morphisms intertwine with the generating boxes.

Then the universal property of $\FreeGS_\Sigma$ reads as follows.

\begin{theorem}
	\label{cdterm_free}
	Let $\Sigma$ be a monoidal signature. Then composing with the canonical hypergraph morphism $\Sigma \to \hyp(\FreeGS_\Sigma)$ induces an equivalence
	\[
		\gsCat(\FreeGS_\Sigma, \cC) \: \cong \: \CatHyp(\Sigma, \hyp(\cC))_{\det}
	\]
	for every gs-monoidal category $\cC$.
\end{theorem}

Most of the remainder of this section is dedicated to the proof.

\begin{example}
	\label{free_monoids}
	Continuing on from \Cref{degenerate}, suppose that $\Sigma$ is the trivial hypergraph consisting of a single wire and no box. Then $\FreeGS_\Sigma$ is equivalent to $\FinSet^\op$, and \Cref{cdterm_free} specializes to the well-known fact that $\FinSet$ is equivalent to the PROP for commutative monoids, a well-known result~\cite{lafont2003algebraic}.
\end{example}

\begin{remark}
	The following proof of \Cref{cdterm_free} is somewhat long-winded, involving a number of induction arguments and case distinctions.
	This is not surprising, since already the degenerate case mentioned in the previous example is nontrivial.

	We had considered an alternative proof involving a reduction to the characterization of free hypergraph categories in~\cite{zanasi2017rewriting}.
	While such an approach indeed seems to be feasible, amounting to an embedding of the free gs-monoidal category $\FreeGS_\Sigma$ generated by a monoidal signature $\Sigma$ into the free hypergraph category generated by $\Sigma$, we have ultimately decided against this method.\footnote{The main reason is that Zanasi's proof has two gaps: first, no argument for the functoriality of Zanasi's purported functor $\gamma$ \cite[Theorem~3.2]{zanasi2017rewriting} is given; second, in contrast to what is stated, the result on amalgamations of categories referenced in the proof of \cite[Corollary~3.4]{zanasi2017rewriting} does not transfer to amalgamations of symmetric monoidal categories in any obvious way due to the interchange law imposing additional equations between morphisms.}
\end{remark}

We now start the proof, showing that the functor from left to right given by the restriction is fully faithful and essentially surjective.

For full faithfulness, suppose that we are given two strong gs-monoidal functors
\[
	F, G \: : \: \FreeGS_\Sigma \longrightarrow \cC,
\]
and consider their adjuncts $F^\flat, G^\flat : \Sigma \to \hyp(\cC)$, which are the composites
\[
	\begin{tikzcd}
		\Sigma \ar{r} & \hyp(\FreeGS_\Sigma) \ar{rr}{\hyp(F), \, \hyp(G)} & & \hyp(\cC).
	\end{tikzcd}
\]
For a gs-monoidal transformation $\alpha : F \Rightarrow G$, naturality at $\copi_w$ and $\del_w$ implies that $\alpha_w$ is deterministic for every generating wire $w$, since $F$ and $G$ preserve copy and discard up to their coherence isomorphisms.
Thus restriction takes values in $\CatHyp(\Sigma,\hyp(\cC))_{\det}$.
It is injective on transformations because every object is a word in the generating wires and monoidality determines $\alpha$ on each word.
For surjectivity, let $\tau : F^\flat \Rightarrow G^\flat$ be an internal natural transformation with deterministic wire components.
For each word $X = w_1 \cdots w_n$, define $\alpha_X$ to be the composite
\[
	\begin{tikzcd}
		F(X) \ar{r}{\cong}
		& \displaystyle\bigotimes_{i=1}^n F(w_i) \ar{rr}{\bigotimes_{i=1}^n \tau_{w_i}}
		& & \displaystyle\bigotimes_{i=1}^n G(w_i) \ar{r}{\cong}
		& G(X),
	\end{tikzcd}
\]
where the outer arrows are the coherence isomorphisms of $F$ and $G$.
We then first check monoidality of $\alpha$ and then its naturality.
For the empty word, this gives the composite $F() \cong I \cong G()$, as required for compatibility with the monoidal units.
For any two words $X$ and $Y$, the structure isomorphisms of $F$ and $G$ fit into the square
\[
	\begin{tikzcd}[column sep=large]
		F(X) \otimes F(Y) \ar{r}{\cong} \ar{d}[swap]{\alpha_X \otimes \alpha_Y}
		& F(XY) \ar{d}{\alpha_{XY}} \\
		G(X) \otimes G(Y) \ar{r}[swap]{\cong}
		& G(XY)
	\end{tikzcd}
\]
Upon plugging in the definition of $\alpha$,
one finds that this commutes by the coherence isomorphisms' associativity and unit laws.
Concerning naturality, for a generating box $\langle b\rangle : X \to Y$, the commutative square defining the box component of $\tau$ in \Cref{internal_hyp} gives
$G(\langle b\rangle)\circ\alpha_X = \alpha_Y\circ F(\langle b\rangle)$ after inserting these coherence isomorphisms.
Naturality at $\copi_w$ and $\del_w$ is precisely the determinism of $\tau_w$, and naturality at symmetry isomorphisms follows from the naturality of the symmetry in $\cC$.
By \Cref{pieces_exist} and the decompositions~\eqref{fd_decomp}--\eqref{fb_decomp} below, every morphism is built from these maps and identities by tensor products and composition, both of which preserve the naturality squares.
Thus $\alpha$ is a monoidal natural transformation.
It restricts to $\tau$ by construction.

We address essential surjectivity over the course of the next few subsections.
We assume that $\cC$ is a strict gs-monoidal category, which is without loss of generality by strictification~\cite[Corollary~4.10]{fongspivak2019bells}.
It is now enough to extend a given hypergraph morphism $f : \Sigma \to \hyp(\cC)$ to a strict gs-monoidal functor $F : \FreeGS_\Sigma \to \cC$.
Since on objects we must necessarily have
\[
	F((\underline{n}, \sigma : \underline{n} \to \Sigma)) = \bigotimes_{i=1}^n f(\sigma(i))
\]
by the requirements of strict monoidality and restriction to $f$, it remains to construct $F$ on every morphism
\begin{equation}
	\label{gs-monoidalsd}
	\alpha \coloneqq
	\begin{tikzcd}
		& G \\
		\underline{m} \ar{ur}{p} & & \underline{n} \ar{ul}[swap]{q} 
	\end{tikzcd}
\end{equation}
We do so by induction on the \newterm{complexity} of $\alpha$, by which we mean the quantity
\[
	C(\alpha) \coloneqq |B(G)| + \sum_{w \in W(G)} \left| \sum_{b \in B(G)} \Iin(b,w) + |q^{-1}(w)| - 1 \right|.
\]
As will become clear in the course of the proof, the idea is that this counts the number of boxes plus the minimal number of occurrences of a copy or discard morphism.

\subsection*{Piece decompositions}

In order to arrive at a gs-monoidal string diagram of lower complexity, we need to factor $\alpha$ into morphisms of smaller complexity.
We will do so by stripping off a ``piece'', similarly as in~\cite[p.~307/8]{corradini1999gsmonoidal}.
Also the independent work~\cite[Section~2.3]{milosavljevic2022rewriting} develops another method for decomposing $\alpha$ into simpler parts.

\begin{definition}
	A \newterm{final wire} is $w \in W(G)$ that is not an input to any box,
	\[
		\Iin(b,w) = 0 \qquad \forall b \in B(G).
	\]
\end{definition}

The following more particular notions will be used to decompose any gs-monoidal string diagram into pieces by induction on the complexity.

\begin{definition}
	A \newterm{parallel wire} is a final wire $w$ with\footnote{By left monogamy, the first equation equivalently states that $w$ is not an output of any box.}
	\[
		|p^{-1}(w)| = 1 = |q^{-1}(w)|.
	\]
	Moreover, a \newterm{piece} of a gs-monoidal string diagram \eqref{gs-monoidalsd} is one of the following three things:
	\begin{enumerate}
		\item A \newterm{final discard} is a final wire $w$ with
			\[
				q^{-1}(w) = \emptyset.
			\]
		\item A \newterm{final copy} is a final wire $w$ together with distinct $i,j \in n$ such that
			\[
				q(i) = q(j) = w.
			\]
		\item A \newterm{final box} is a box $b \in B(G)$ such that each of its output wires $w \in \Iout(b)$ is final and satisfies $|q^{-1}(w)| = 1$.
	\end{enumerate}
\end{definition}

Intuitively, a final box is one whose outputs connect to output interfaces without further processing.
In particular, every box without any inputs or outputs is a final box.

\begin{lemma}
	\label{pieces_exist}
	%Every nonempty gs-monoidal string diagram has a piece.
	Every non-permutation gs-monoidal string diagram has a piece.%final discard, final copy or final box.
\end{lemma}

\begin{proof}
	%The second statement implies the first since the gs-monoidal string diagrams which consist only of parallel wires are precisely the permutations.
	Assume that neither of the three kinds of pieces exists in a gs-monoidal string diagram~\eqref{gs-monoidalsd}.
	Then every final wire $w$ satisfies $|q^{-1}(w)| = 1$, and by left monogamy is either an output of a box or a parallel wire.
	In particular, since by assumption there is at least one non-parallel wire, there must at least be one box.
	And since there is no final box, it follows that every box has an output wire that is not final.
	For every box $b$, we can therefore choose a ``successor'' box $b'$ defined as any box that has a non-final output wire of $b$ as an input.
	If we now start at an arbitrary box and consider its sequence of successors, by finiteness of $G$ we end up with a cycle that contradicts the acyclicity of $G$.
\end{proof}

The permutations (symmetry isomorphisms) in $\FreeGS_\Sigma$ are exactly those morphisms represented by cospans with no boxes and for which both legs are bijections on nodes. Equivalently, they are the morphisms of complexity zero.
On these, the value of $F$ is clear: since $F$ is required to be a symmetric monoidal functor, this value must be the corresponding permutation in $\cC$.

Assume now that $F$ has already been defined on all gs-monoidal string diagrams of lower complexity.
We then define it on the gs-monoidal string diagram under consideration $\alpha$ by stripping off a piece, applying $F$ to the remaining gs-monoidal string diagram of lower complexity $\alpha'$, and extending $F$ to $\alpha$ in the unique way.
Delegating the question of well-definedness to further down, we do this by distinguishing the type of piece.

\begin{enumerate}
	%\item If there is a parallel wire $w$, then the gs-monoidal string diagram decomposes, up to permutation of input and output interfaces, as
	%	\begin{equation}
	%		\label{pw_decomp}
	%		\left(
	%			\begin{tikzcd}
	%				& \underline{1} \\
	%				\underline{1} \ar{ur}{\id} & & \underline{1} \ar{ul}[swap]{\id}
	%			\end{tikzcd}
	%		\right) \otimes \left(
	%			\begin{tikzcd}
	%				& G' \\
	%				\underline{m-1} \ar{ur}{p'} & & \underline{n-1} \ar{ul}[swap]{q'}
	%			\end{tikzcd}
	%		\right)
	%	\end{equation}
	%	Here, $G'$ coincides with $G$ except in that $w$ has been removed, and likewise the input and output interfaces are such that $p^{-1}(w)$ and $q^{-1}(w)$ has been removed, respectively.
	%	Since $w$ is neither an input or an output of any box, it is clear that this tensor product reproduces the given gs-monoidal string diagram up to permutations of inputs and outputs.
	%	Writing $\alpha'$ for the gs-monoidal string diagram corresponding to the second factor, it is clear that $C(\alpha') = C(\alpha) - 1$, so that $F$ is already defined on $\alpha'$.
	%
	%	Modulo the relevant permutations on inputs and outputs, the value of $F$ on the gs-monoidal string diagram can now be defined as
	%	\[
	%		F(\alpha) \coloneqq F(\id_{f(\sigma(w))}) \otimes F(\alpha'),
	%	\]
	%	and it is clear that this is the only possibility.
	\item If there is a final discard $w$ in $\alpha$, then up to permutation of output interfaces, $\alpha$ can be decomposed into the composite cospan of
		\begin{equation}
			\label{fd_decomp}
			\begin{tikzcd}
				& G & & \underline{n+1} \\
				\underline{m} \ar{ur}{p} & & \underline{n+1} \ar{ul}[swap]{q'} \ar{ur}{\id} & & \underline{n} \ar{ul}
			\end{tikzcd}
		\end{equation}
		Here, the unlabeled arrow is the trivial inclusion, and $q'$ is given by $q$ extended by $n + 1 \longmapsto w$.
		This amounts to the creation of a new output interface for the wire $w$ that was discarded in $\alpha$. 
		Then the composite gs-monoidal string diagram coincides with $\alpha$ by construction.
		With $\alpha'$ referring to the first cospan above (with the same labeling $\sigma$), we have $C(\alpha') = C(\alpha) - 1$, since the summand associated to $w$ is $0$ in $\alpha'$ and $1$ in $\alpha$.

		We now define
		\[
			F(\alpha) \coloneqq \left(\id \otimes \del_{f(\sigma(w))} \right) \circ F(\alpha').
		\]
		This is the only possibility (subject to our requirements) since the second gs-monoidal string diagram above is equal to $\id_{\underline{n}} \otimes \del_{\sigma(w)}$ in $\FreeGS_\Sigma$.
	\item The case of a final copy $w$ in $\alpha$ is similar. By permuting output interfaces we can assume that $q(n-1) = q(n) = w$.
		Then the gs-monoidal string diagram $\alpha$ decomposes as the composite
		\begin{equation}
			\label{fc_decomp}
			\begin{tikzcd}
				& G & & \underline{n-1} \\
				\underline{m} \ar{ur}{p} & & \underline{n-1} \ar{ul}[swap]{q'} \ar{ur}{\id} & & \underline{n} \ar{ul}
			\end{tikzcd}
		\end{equation}
		Here, the unlabeled arrow maps $n \longmapsto n-1$ and acts as the identity otherwise, and $q'$ is the restriction of $q$ to $\underline{n-1}$.
		That the composite recovers $\alpha$ and that the complexity of $\alpha'$ has decreased can be seen as in the previous case.

		We can therefore define
		\[
			F(\alpha) \coloneqq \left(\id \otimes \copi_{f(\sigma(w))} \right) \circ F(\alpha'),
		\]
		and this is the only possibility since the second gs-monoidal string diagram above decomposes as $\id_{\underline{n-2}} \otimes \copi_{\sigma(w)}$ in $\FreeGS_\Sigma$.
	\item If $\alpha$ has a final box $b$ of arity $(k,\ell)$, then by output permutations we can assume that the outputs of $b$ are exactly the wires connected to the output interfaces $n-\ell+1,\ldots,n$ in the same order.
		The gs-monoidal string diagram $\alpha$ now decomposes as
		\begin{equation}
			\label{fb_decomp}
			\begin{tikzcd}
				& G' & & \underline{n-\ell} + \overline{\sigma(b)} \\
				\underline{m} \ar{ur}{p} & & \underline{n-\ell+k} \ar{ul}[swap]{q'} \ar{ur} & & \underline{n} \ar{ul}
			\end{tikzcd}
		\end{equation}
		Here, the second cospan is such that it represents the gs-monoidal string diagram $\id_{\underline{n-\ell}} \otimes \langle \sigma(b) \rangle$, and $G'$ consists of $G$ with the box $b$ and the output wires of $b$ removed. The right leg $q'$ can be defined as $q$ on $\underline{n-\ell}$, and as mapping to the sequence of input wires\footnote{Note that these need not be all distinct.} of $b$ on the remaining output interfaces.
		The gs-monoidal string diagram $\alpha'$ represented by the left cospan has complexity $C(\alpha') = C(\alpha) - 1 < C(\alpha)$.

		We can now define
		\[
			F(\alpha) \coloneqq \left(\id \otimes f(\sigma(b)) \right) \circ F(\alpha'),
		\]
		and this is the only possibility due to the requirement that $F(\langle \sigma(b) \rangle) = f(\sigma(b))$.
\end{enumerate}

Stripping off pieces repeatedly will completely decompose $\alpha$ into pieces eventually, since the complexity decreases in every step.
Thus if one starts out by successively removing final discards, copies and boxes, then this will ultimately result in a permutation, on which $F$ is trivially defined.

\subsection*{Well-definedness}

Showing that $F$ is well-defined requires two steps.
First, we need to check that the above definitions of $F(\alpha)$ are independent of the particular permutations of inputs and outputs that have been used to arrive at the specified forms~\eqref{fd_decomp}--\eqref{fb_decomp} in each step. This follows upon proving in the same induction that $F$ satisfies functoriality whenever one of the morphisms involved is a permutation. This straightforwardly reduces in each case to precomposing the legs $p'$ and $q'$ by permutation on interfaces.\footnote{This is where the commutativity of the comonoid structures in $\cC$ enters (in the final copy case).} We leave the detailed arguments to the reader.

Second, it must be shown that the above definitions are independent of the choice of piece.
We thus assume that we are given two different pieces of $\alpha$, and argue that applying the above prescriptions results in the same morphism $F(\alpha)$.

After permutations of output interfaces, each case amounts to a factorization $\alpha = (\id \otimes \gamma) \circ \alpha'$, where $\gamma$ corresponds to either a single discard, copy or box.
Then the well-definedness proof reduces to the following two exclusive cases, also using induction on the complexity.

\begin{itemize}
	\item The output interfaces appearing in the two pieces are disjoint.

		Then both pieces can be stripped off at once.
		After another permutation of output interfaces, this results in a factorization
		\[
			\alpha = (\gamma_1 \otimes \id \otimes \gamma_2) \circ \alpha'',
		\]
		where $\gamma_1$ and $\gamma_2$ each are a single discard, copy or box, and such that the two sides of the resulting equation
		\[
			(\gamma_1 \otimes \id) \circ \left[ (\id \otimes \gamma_2) \circ \alpha'' \right] = (\id \otimes \gamma_2) \circ \left[ (\gamma_1 \otimes \id) \circ \alpha'' \right]
		\]
		are the two factorizations used in each definition of $F(\alpha)$. Since
		\begin{align*}
			F \left( (\gamma_1 \otimes \id) \circ \alpha'' \right) = (F(\gamma_1) \otimes \id) \circ F(\alpha'') \\
			F \left( (\id \otimes \gamma_2) \circ \alpha'' \right) = (\id \otimes F(\gamma_2)) \circ F(\alpha'')
		\end{align*}
		holds by the assumed well-definedness at lower complexity, the interchange law in $\cC$ proves that the two definitions of $F(\alpha)$ are equal as morphisms of $\cC$.
	\item Both pieces are final copies with overlapping output interfaces.

		Then we can assume without loss of generality that these output interfaces are exactly $n-2$ to $n$.
		It is straightforward to see that the well-definedness now holds because of the coassociativity of $\copi_{f(\sigma(w))}$ in $\cC$.
\end{itemize}

\subsection*{Functoriality}

We now prove the functoriality
\[
	F(\beta \circ \alpha) = F(\beta) \circ F(\alpha)	
\]
for any pair of composable gs-monoidal string diagrams $\alpha : \underline{\ell} \to \underline{m}$ and $\beta : \underline{m} \to \underline{n}$.
If $\beta$ is a permutation, then the claim was already established above.
The next case to consider is when $\beta$ has complexity one.

\begin{lemma}
	\label{functoriality_one}
	If $\beta$ has complexity one, then $F(\beta \circ \alpha) = F(\beta) \circ F(\alpha)$ for every $\alpha$ pre-composable with $\beta$.
\end{lemma}

\begin{proof}
	Since $\beta$ has complexity one, it consists of a single box, a copy or a discard tensored with identities, up to permutations of interfaces.
	We suppress these permutations since functoriality for them has already been established.
	If $\beta$ consists of a box tensored with identities, this box is final in $\beta \circ \alpha$, and stripping it off recovers $\alpha$.
	In this case, the claim is clear by definition of $F$.

	Now suppose that $\beta$ copies or discards an output interface of $\alpha$, and let $w$ be the corresponding wire.
	We argue by induction on the complexity $C(\alpha)$.
	First suppose that $w$ is final in $\alpha$.
	Then if $\beta$ copies any or discards the only output interface of $\alpha$ on $w$, 
	then stripping the resulting piece from $\beta \circ \alpha$ recovers $\alpha$.
	However, if $\beta$ discards one of several output interfaces on $w$,
  	then pairing it with any surviving one gives a final copy in $\alpha$ whose removal yields $\beta \circ \alpha$ again,
	and the claim follows from the counit law $(\id \otimes \del) \circ \copi = \id$ in $\cC$.

	The remaining case is when $w$ is not final in $\alpha$,
	meaning that it is an input to one or more boxes in $\alpha$.
	By \Cref{pieces_exist}, we can choose a piece in $\alpha$.
	The wire of a final copy or discard and the output wires of a final box are all final in $\alpha$, so each is distinct from $w$.
	Since composing with $\beta$ changes only the output interfaces on $w$, the chosen piece is also a piece in $\beta \circ \alpha$.
	Stripping it off from $\alpha$ gives a diagram $\alpha'$ of lower complexity, and stripping it off from $\beta \circ \alpha$ gives the same copy or discard applied to the chosen interface in $\alpha'$.
	If the piece is a box with $w$ as an input, stripping it off exposes a separate interface for each such input.
	The induction hypothesis applies to $\alpha'$, and reattaching the piece gives the desired equality by the interchange law in $\cC$, since the two operations act on disjoint tensor factors.
\end{proof}

We now proceed by induction on $C(\beta)$, assuming $C(\beta)>1$. By \Cref{pieces_exist} and the construction of $F$, there is a decomposition
\[
	\beta = \gamma \circ \beta'
\]
where $\gamma$ has complexity one and $\beta'$ has lower complexity than $\beta$.
Applying \Cref{functoriality_one} to $\gamma$ gives
\[
	F(\gamma \circ \beta') = F(\gamma) \circ F(\beta'), \qquad F(\gamma \circ \beta' \circ \alpha) = F(\gamma) \circ F(\beta' \circ \alpha).
\]
Since $F(\beta' \circ \alpha) = F(\beta') \circ F(\alpha)$ by the induction hypothesis, we obtain
\[
	F(\beta \circ \alpha) = F(\gamma \circ \beta' \circ \alpha) = F(\gamma) \circ F(\beta') \circ F(\alpha) = F(\beta) \circ F(\alpha),
\]
which is the induction step.

\subsection*{Monoidality}

We now sketch the proof of the strict monoidality
\[
	F(\alpha \otimes \beta) = F(\alpha) \otimes F(\beta)
\]
for arbitrary gs-monoidal string diagrams $\alpha$ and $\beta$. By the interchange law in both the domain and codomain categories as well as by functoriality, it is enough to consider the case where $\alpha$ or $\beta$ is an identity morphism, so suppose $\beta = \id$.
In this case, the claim follows by another straightforward induction argument on the complexity of $\alpha$, doing the induction step separately for each type of piece.

\subsection*{Preservation of gs-monoidal structure}

It remains to prove that $F$ preserves the gs-monoidal structure. This holds on a copy morphism of type $\copi_{\underline{1}}$ by construction in the final copy piece case.
The general case then follows by the multiplicativity of $\copi$ across tensor products of objects.
The analogous argument for discard then establishes the claim.

\medskip

In conclusion, $F : \FreeGS_\Sigma \to \cC$ is a strict gs-monoidal functor extending the given hypergraph morphism $f : \Sigma \to \hyp(\cC)$.
This finishes the proof of \Cref{cdterm_free}.

\subsection*{$1$-Categorical universal property}

We also obtain two versions of a $1$-categorical universal property analogous to the ones considered in~\cite{corradini1999gsmonoidal}.
While the $2$-categorical universal property given in \Cref{cdterm_free} is arguably the most relevant and useful one in practice, the following $1$-categorical one has the advantage of being easier to understand.
It also characterizes $\FreeGS_\Sigma$ up to isomorphism (rather than merely up to equivalence).
In the following statement, we consider $\hyp(\cC)$ as a mere hypergraph again.

\begin{corollary}
	\label{cdterm_free1}
	Let $\cC$ be a strict gs-monoidal category whose monoid of objects is the free monoid with generating set $W(\Sigma)$.	
	Then restricting along the hypergraph morphism $\Sigma \to \hyp(\FreeGS_\Sigma)$ induces a bijection between:
	\begin{itemize}
		\item The  strict gs-monoidal functors $\FreeGS_\Sigma \to \cC$,
		\item The  hypergraph morphisms $\Sigma \to \hyp(\cC)$.
	\end{itemize}
	This bijection restricts to a bijection between:
	\begin{itemize}
		\item The identity-on-objects strict gs-monoidal functors $\FreeGS_\Sigma \to \cC$.
		\item The identity-on-wires hypergraph morphisms $\Sigma \to \hyp(\cC)$.
	\end{itemize}
\end{corollary}

The second claim is an immediate special case of the first.
These statements do not follow directly from \Cref{cdterm_free}, because the latter characterizes $\FreeGS_\Sigma$ only up to equivalence.
However, it follows by the same proof, since the gs-monoidal functor $F$ constructed as the extension of the hypergraph morphism $f$ is strict, and our construction also demonstrates its uniqueness since the given definition of $F(\alpha)$ is the only possibility.

\section{The bloom-circuitry factorization}
\label{bloom_circuitry}

For any monoidal signature $\Sigma$, we show that the free gs-monoidal category $\FreeGS_\Sigma$ has a canonical factorization system.

\begin{definition}
	\label{defLCDRCD}
	A morphism in $\FreeGS_\Sigma$ represented by a gs-monoidal string diagram
	\begin{equation}
		\label{cd_string}
		\begin{tikzcd}
			& G \\
			\underline{m} \ar{ur}{p} & & \underline{n} \ar{ul}[swap]{q}
		\end{tikzcd}
	\end{equation}
	is:
	\begin{enumerate}
		\item \newterm{pure bloom} if the right leg $q$ is a bijection (on wires).
		\item \newterm{pure circuitry} if $G$ has no boxes.
	\end{enumerate}
\end{definition}

For example, the string diagram depicted in~\Cref{bloom_example} is a pure bloom. The point is that every wire becomes an overall output in exactly one way.\footnote{This nicely matches up with the traditional random variable setting of probability theory, and in particular with Bayesian networks: every random variable in a Bayesian network forms part of the overall joint distribution, and is therefore wired to exactly one output interface in our sense.} A pure circuity morphism in turn is a gs-monoidal string diagram that contains at most copies and discard operations. 

The ``bloom'' terminology is inspired by and generalizes the one of~\cite{fullwood2021loss}, where blooms of a particularly simple form were considered. Our notion is similar to a condition used at~\cite[Lemma~4.7]{corradini1999gsmonoidal}, which in our terminology amounts to surjectivity of $q$.

\begin{figure}
	\centering
	\input{bloom2}
	\caption{A gs-monoidal string diagram that is a pure bloom since every wire becomes an output in exactly one way.
		The gs-monoidal string diagram of \Cref{cd_string_hypergraph} factors into this pure bloom followed by pure circuitry.}
	\label{bloom_example}
\end{figure}
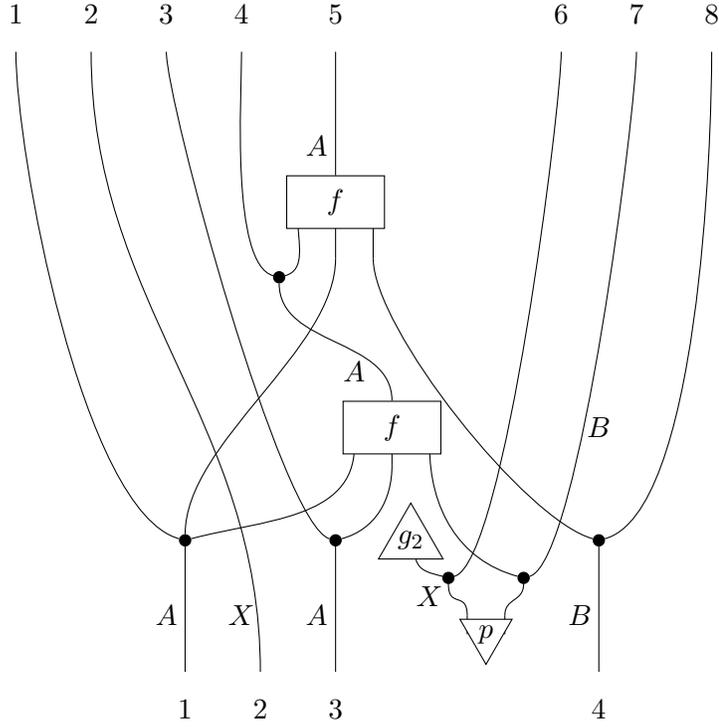

\begin{lemma}
	The pure blooms form a symmetric monoidal subcategory of $\FreeGS_\Sigma$.
\end{lemma}

\begin{proof}
	In a sequential composition~\eqref{pushout_comp}, if $q$ is a bijection on wires then so is $u$ (since pushouts in $\FinHyp / \Sigma$ are computed as pushouts in $\Set$ on the sets of wires).
	Since $s$ is a bijection on wires by assumption as well, it follows that $us$ is too, and hence the composite cospan is a pure bloom.

	The preservation of pure blooms by parallel composition~\eqref{parallel_comp} is obvious, as is the fact that the symmetry isomorphisms are pure blooms.
\end{proof}

We denote this symmetric monoidal subcategory by $\Bloom_\Sigma$.
Of course, the gs-monoidal structure morphisms~\eqref{freecd_cd} are not pure blooms, so that $\Bloom_\Sigma$ is not a gs-monoidal category.
On the other hand, the pure circuitry morphisms form a gs-monoidal subcategory which coincides with $\FreeGS_{\underline{W(\Sigma)}}$, which is equivalent to $(\FinSet / W(\Sigma))^\op$ (\Cref{degenerate}).

\begin{lemma}
	The following are equivalent for a morphism $\alpha$ in $\FreeGS_\Sigma$:
	\begin{enumerate}[label=(\roman*)]
		\item\label{pureboth} $\alpha$ is pure bloom and pure circuitry.
		\item\label{zeroC} $C(\alpha) = 0$.
		\item\label{isperm} $\alpha$ is a permutation.
		\item\label{isiso} $\alpha$ is an isomorphism.
	\end{enumerate}
\end{lemma}

\begin{proof}
	By left monogamy, a pure circuitry morphism is necessarily such that its left leg is an isomorphism.
	Assuming \ref{pureboth}, we therefore obtain that both legs are isomorphisms.
	Together with $\alpha$ having no boxes, we obtain \ref{zeroC}.
	The equivalence of \ref{zeroC} and \ref{isperm} was already discussed in the previous section.
	The implication \ref{isperm}$\Rightarrow$\ref{isiso} is trivial.
	Assuming \ref{isiso}, it is clear that $\alpha$ must be pure circuitry.
	Since then the left leg is an isomorphism by left monogamy, the isomorphism assumption implies that so is the right leg, making $\alpha$ pure bloom as well.
\end{proof}

\begin{proposition}
	\label{factorization}
	The pure blooms and the pure circuitry morphisms form an orthogonal factorization system for $\FreeGS_\Sigma$:
	\begin{enumerate}
		\item Every morphism factors into a pure bloom followed by pure circuitry.
		\item This factorization is unique up to a unique permutation of wires.
	\end{enumerate}
\end{proposition}

\begin{proof}
	Given any gs-monoidal string diagram~\eqref{cd_string}, we factor it in $\cospan(\FinHyp / \Sigma)$ as
	\[
		\begin{tikzcd}
			& G & & \underline{W(G)} \\
			\underline{m} \ar{ur}{p} & & \underline{W(G)} \ar{ul}[swap]{\id} \ar{ur}{\id} & & \underline{n} \ar{ul}[swap]{q}
		\end{tikzcd}
	\]
	The two constituent cospans become gs-monoidal string diagrams upon choosing any total ordering on $W(G)$, or equivalently any hypergraph isomorphism $\underline{W(G)} \cong \underline{|W(G)|}$, and considering $\underline{|W(G)|}$ an object of $\FinHyp / \Sigma$ by equipping it with the labeling inducing by this isomorphism.
	The left cospan is then a pure bloom and the right one pure circuitry by construction.

	Concerning uniqueness, any factorization into pure bloom and pure circuitry results in two cospans such that the two middle legs are bijections on wires.
	This implies that the ambiguity in the factorization is exactly the choice of isomorphism $\underline{W(G)} \cong \underline{|W(G)|}$.
	Choosing a different one amounts precisely to conjugation by a permutation isomorphism identifying the two intermediate objects.
	The uniqueness of the permutation is obvious, since in particular it must commute with the right leg of the pure bloom cospan.

	We therefore indeed have an orthogonal factorization system by~\cite[Proposition~14.7]{adamek2006cats}, since the isomorphisms in $\FreeGS_\Sigma$ are exactly the permutations.
\end{proof}

In particular, \Cref{factorization} implies that the two subcategories define a distributive law of colored PROPs~\cite[Theorem~4.6]{lack2004props}.

\section{Free Markov Categories}
\label{freemarkov}

Given a monoidal signature $\Sigma$, we now investigate the free Markov category $\FreeMarkov_\Sigma$, satisfying an analogous $2$-categorical universal property as $\FreeGS_\Sigma$.

\begin{example}
	\label{cd_vs_markov}
	$\FreeGS_\Sigma$ is a Markov category if and only if the monoidal signature $\Sigma$ has no boxes, $B(\Sigma) = \emptyset$.
	
	Indeed the ``if'' direction is obvious since we then are in the degenerate case of \Cref{degenerate}, where the monoidal structure is cartesian, and in particular the monoidal unit is terminal.
	For the ``only if'' direction, suppose that $\Sigma$ contains a box $b \in B(\Sigma)$. Then consider $\langle b\rangle$, the gs-monoidal string diagram consisting of $b$ alone. 
	Then discarding all outputs of $b$ (if any) produces a gs-monoidal string diagram which still has one box, contradicting the naturality of discarding.
\end{example}

In order to construct $\FreeMarkov_\Sigma$, we can simply start with $\FreeGS_\Sigma$ and quotient by the naturality of discarding.
However, in practice it will be more convenient to choose one representative of each equivalence class.
We thus introduce the following normal form.

\begin{definition}
	Let a gs-monoidal string diagram	
	\[
		\alpha \coloneqq
		\begin{tikzcd}
			& G \\
			\underline{m} \ar{ur}{p} & & \underline{n} \ar{ul}[swap]{q} 
		\end{tikzcd}
	\]
	be given.
	\begin{enumerate}[label=(\roman*)]
		\item A box $b \in B(G)$ is \newterm{eliminable} if each one of its output gets discarded: for every $w \in W(G)$,
			\[
				\Iout(b,w) > 0 \quad \Longrightarrow \quad q^{-1}(w) = \emptyset \quad \land \quad \forall b' \in B(G) : \: \Iin(b', w) = 0.
			\]
		
		\item $\alpha$ is \newterm{normalized} if it has no eliminable boxes.
	\end{enumerate}
\end{definition}

Every gs-monoidal string diagram can be turned into a normalized one by repeatedly removing eliminable boxes and their output wires (if any), while all other boxes and wires remain unchanged.
Removing a box like this decreases the number of boxes, and therefore repeatedly eliminating boxes eventually terminates.
If two distinct boxes $b_1, b_2 \in B(G)$ are both eliminable, then they can be removed in either order, which implies confluence.
Therefore the process is guaranteed to terminate in a unique normalized form. We call it the \newterm{normalization}.
For a simple example, the normalization of a single effect (in the monoidal signature of \Cref{monoidal_sig}) is given by
\[
	\input{effect}
\]

\begin{definition}
	\label{cat_markov_diagrams}
	Given a hypergraph $\Sigma$, $\FreeMarkov_\Sigma$ is the category with:
	\begin{itemize}
		\item Objects are the pairs $(n, \sigma : \underline{n} \to \Sigma)$ with $n \in \N$.
		\item Morphisms are the isomorphism classes of normalized gs-monoidal string diagrams (\Cref{cd_diagram}).
		\item Composition is defined as in $\FreeGS_\Sigma$ and subsequent normalization.
	\end{itemize}
\end{definition}

By \Cref{cd_vs_markov}, a composite of normalized gs-monoidal string diagrams need not be normalized, and a separate normalization step is required in general.
However, if $\alpha$ and $\beta$ are normalized gs-monoidal string diagrams, then also $\alpha \otimes \beta$ is already normalized: if $\alpha \otimes \beta$ contains an eliminable box, then this same box is also eliminable in $\alpha$ or $\beta$.
Thus the monoidal structure of $\FreeGS_\Sigma$ directly restricts to $\FreeMarkov_\Sigma$, and the same applies to the gs-monoidal structure.
Hence $\FreeMarkov_\Sigma$ is clearly a gs-monoidal category.

\begin{lemma}
	$\FreeMarkov_\Sigma$ is a Markov category.
\end{lemma}

\begin{proof}
	It is enough to show that the monoidal unit $\underline{0}$ is terminal.
	If $\alpha : \underline{n} \to \underline{0}$ is any normalized string diagram, then it cannot have any boxes, since every such box would be eliminable.
	Thus $\alpha$ is pure circuitry, and therefore equal to the gs-monoidal string diagram which simply discards all inputs.
\end{proof}

\begin{lemma}
	Mapping every gs-monoidal string diagram to its normalization is an identity-on-objects strict gs-monoidal functor
	\[
		\norm \: : \: \FreeGS_\Sigma \longrightarrow \FreeMarkov_\Sigma.
	\]
\end{lemma}

\begin{proof}
	We need to prove functoriality and monoidality, which amounts to the equations
	\begin{align*}
		\norm(\alpha \circ \beta) & \, = \, \norm(\norm(\alpha) \circ \norm(\beta)),	\\
		\norm(\alpha \otimes \beta) & \, = \, \norm(\alpha) \otimes \norm(\beta),
	\end{align*}
	whenever these make sense.
	By the normal form property of normalization, for functoriality it is enough to show that $\alpha \circ \beta$ reduces to $\norm(\alpha) \circ \norm(\beta)$, which follows by induction on the number of boxes, noting that any eliminable box in $\alpha$ or $\beta$ is still eliminable in $\alpha \circ \beta$, and noting that removing an eliminable box and its output wires does not interfere with the pushout composition.
	The monoidality follows similarly by showing that $\alpha \otimes \beta$ reduces to $\norm(\alpha) \otimes \norm(\beta)$.
\end{proof}

Under the correspondence of \Cref{cdterm_free1}, the functor $\norm$ corresponds to the hypergraph morphism $\Sigma \to \hyp(\FreeMarkov_\Sigma)$ given by the composite
\[
	\begin{tikzcd}[column sep=4pc]
		\Sigma \ar{r} & \hyp(\FreeGS_\Sigma) \ar{r}{\hyp(\norm)} & \hyp(\FreeMarkov_\Sigma).
	\end{tikzcd}
\]
This identity-on-wires morphism sends every box $b$ to $\norm{\langle b\rangle}$, which is $\langle b \rangle$ itself unless $b$ has no outputs, in which case it gets mapped to the gs-monoidal string diagram that discards all its inputs.

Here is the Markov category version of \Cref{cdterm_free}.

\begin{theorem}
	\label{markov_free}
	Let $\Sigma$ be a monoidal signature. Then composing with the canonical hypergraph morphism $\Sigma \to \hyp(\FreeMarkov_\Sigma)$ induces an equivalence
	\[
		\gsCat(\FreeMarkov_\Sigma, \cC) \: \cong \: \CatHyp(\Sigma, \hyp(\cC))_{\det}
	\]
	for every Markov category $\cC$.
\end{theorem}

\begin{proof}
	By \Cref{cdterm_free}, it is enough to show that composing with $\norm$ induces an equivalence
	\[
		\begin{tikzcd}[column sep=1.5pc]
			\gsCat(\FreeMarkov_\Sigma, \cC) \ar{r}	& \gsCat(\FreeGS_\Sigma, \cC)
		\end{tikzcd}
	\]
	for every Markov category $\cC$.
	Faithfulness is obvious since $\norm$ is identity-on-objects.
	Fullness holds by the same reason together with fullness of $\norm$, which enters in the verification of naturality.
	For essential surjectivity, it is enough to argue that every strong gs-monoidal functor $F : \FreeGS_\Sigma \to \cC$ factors through $\norm$, or equivalently that $F(\alpha) = F(\beta)$ whenever $\norm(\alpha) = \norm(\beta)$.
	We can assume without loss of generality that $\alpha$ and $\beta$ differ by a single reduction step, consisting of the elimination of an eliminable box $b$ in $\alpha$.
	We can then write $\alpha$ as a composite gs-monoidal string diagram
	\[
		\alpha = \left[ \id \otimes (\del \circ \langle b \rangle) \right] \circ \alpha'
	\]
	in such a way that
	\[
		\beta = \left[ \id \otimes \del \right] \circ \alpha'.
	\]
	The naturality of discarding in $\cC$ then implies $F(\alpha) = F(\beta)$.
\end{proof}

Proceeding by essentially the same arguments, we also obtain the following.

\begin{corollary}
	\label{markov_free1}
	Let $\cC$ be a strict Markov category whose monoid of objects is the free monoid with generating set $W(\Sigma)$.	
	Then restricting along $\Sigma \to \hyp(\FreeMarkov_\Sigma)$ induces a bijection between:
	\begin{itemize}
		\item The strict gs-monoidal functors $\FreeMarkov_\Sigma \to \cC$.
		\item The hypergraph morphisms $\Sigma \to \hyp(\cC)$.
	\end{itemize}
	This bijection restricts to a bijection between:
	\begin{itemize}
		\item The identity-on-objects strict gs-monoidal functors $\FreeMarkov_\Sigma \to \cC$.
		\item The identity-on-wires hypergraph morphisms $\Sigma \to \hyp(\cC)$.
	\end{itemize}
\end{corollary}

\section{GS-monoidal string diagrams as generalized causal models}
\label{causal}

Here, we briefly discuss how free Markov categories provide a notion of causal model which is both more general and more intuitive than the traditional one.
We refer to~\cite{fritz2023dseparation} for more details and examples as well as a proof of the \emph{d-separation criterion} for Bayesian networks in categorical terms.\footnote{See also also refer to the works of Fong~\cite{fong2013causal} and Yin and Zhang~\cite{yinzhang2022causal} for earlier considerations on the relation to Bayesian networks.}
Besides the greater generality of the categorical formalism, another important advantage is that this criterion can be formulated as a statement about topological connectedness of string diagrams, which makes it more intuitive than the traditional one.

Traditionally, causal models are defined as \emph{directed acyclic graphs}.
A \emph{Bayesian network} is an assignment of random variables to the vertices in such a way that their dependencies suitably match the causal structure defined by the graph~\cite{pearl2009causality}.
In the following, no prior knowledge of Bayesian networks will be needed.
Instead, let us focus on our proposed definition.

\begin{definition}
	A \newterm{generalized causal model} is a normalized gs-monoidal string diagram
	
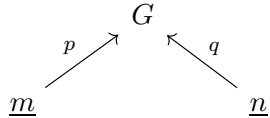
\[\begin{tikzcd}
		& G \\
		\underline{m} \ar{ur}{p} & & \underline{n} \ar{ul}[swap]{q}
	\end{tikzcd}\]
	in a monoidal signature $\Sigma$, or equivalently a morphism in $\FreeMarkov_\Sigma$, such that the right leg $q$ is injective.
\end{definition}

The injectivity condition for $q$ amounts to the requirement that every wire can become an overall output in at most one way.
This guarantees that every output can be interpreted as its own random variable; while if $q$ is not injective, then several outputs represent copies of the same random variable, which is an uninteresting case.

\begin{figure}
	\input{markov_string_hypergraph}
	\caption{A normalized gs-monoidal string diagram in the monoidal signature $\Sigma$ of \Cref{monoidal_sig}.}
	\label{markov_string_hypergraph}
\end{figure}
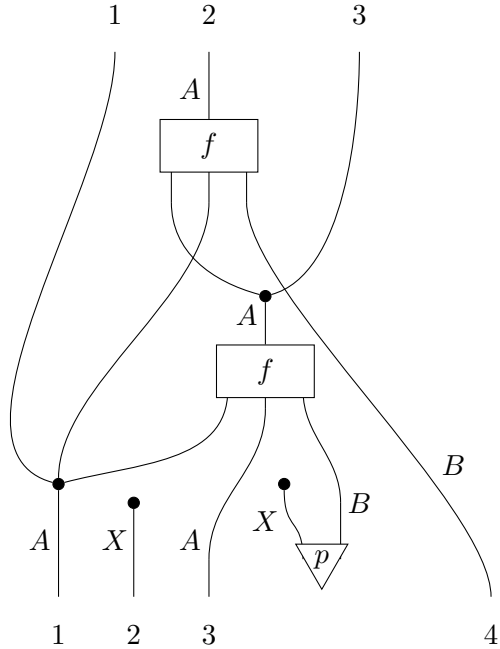

For example, consider the normalized gs-monoidal string diagram depicted in \Cref{markov_string_hypergraph}, and denote it by $\alpha$.
It depicts a generalized causal model for a Markov kernel with four inputs and three outputs, where the three outputs are of the same type as the first and third input.
In other words, it is a statistical model for Markov kernels of the form
\[
	M \: : \: A' \otimes X' \otimes A' \otimes B' \longrightarrow A' \otimes A' \otimes A',
\]
where $A'$, $B'$ and $X'$ are fixed measurable spaces (or just finite sets).
We say that such $M$ is \newterm{compatible} with the model $\alpha$ if there exists a Markov functor $F : \FreeMarkov_\Sigma \to \Stoch$ such that 
\[
	F(A) = A', \qquad F(B) = B', \qquad F(X) = X'
\]
and\footnote{This equation must be understood as holding modulo the coherence isomorphisms of $F$, which make the domain of the left-hand side, which is $F(A \otimes X \otimes A \otimes B)$, match up with the domain of the right-hand side, which is $F(A) \otimes F(X) \otimes F(A) \otimes F(B)$.}
\[
	F(\alpha) = M.
\]
This last equation is the main compatibility condition, saying that $M$ needs to be implementable using a causal structure of the given shape $\alpha$.
It formalizes the intuitive idea of the gs-monoidal string diagram $\alpha$ depicting the actual causal mechanisms: 
compatibility means that one can find a particular distribution $p$ and a Markov kernel $f$ such that plugging these pieces together, using $\alpha$ as a blueprint, results in $M$.

\begin{remark}
	Of course, there is nothing specific to Markov kernels in the category $\Stoch$ here: the same definitions can be used as the definition of compatibility between a morphism in any Markov category and the specification of a causal structure by a gs-monoidal string diagram.
\end{remark}

Defined like this, generalized causal models have the following advantage over causal models as usually defined in terms of directed acyclic graphs:
\begin{itemize}
	\item Latent variables are built in from the start.
		For example in \Cref{markov_string_hypergraph}, the left output of $p$ does not connect to an output interface and is therefore not part of the Markov kernel that is being modeled.
		In the traditional setting of Bayesian networks, this would be considered a latent variable.
		The analogous statement applies to the output of the first instance of $f$.

		The gs-monoidal string diagrams that do not have any latent variables in this sense---and do not have any unnecessary duplications either---are exactly the pure blooms (\Cref{defLCDRCD}).
	\item It is possible and natural to encode the requirement for the same causal mechanism to appear multiple times, as in the Markov kernel $f$ appearing twice in \Cref{markov_string_hypergraph}.
		For example, this is relevant for the inflation technique in causal inference~\cite{wolfe2019inflation,gitton2022inflation}. It seems like a natural condition in general, since real-world systems often contain one and the same mechanisms several times.
	\item Generalized causal models are models for \emph{Markov kernels} rather than probability distributions.
		In other words, they naturally allow for the consideration of ``control'' quantities which do not have any particular distribution, but rather serve as overall input variables.
		This is a situation often considered for example in physics in the context of Bell's theorem, where the ``settings'' of the two ``parties'' do not have a fixed distribution~\cite{brunner2013bell}. It is sometimes implemented in the conventional formalism by an additional annotation that writes observed variables in circles and control variables in squares~\cite{shpitser2014nested}.
	\item GS-monoidal string diagrams provide an arguably more intuitive representation of the information flow than a directed acyclic graph does. 
		One simply needs to represent each quantity as a wire, and work out the mechanisms by which particular variables get transformed into other ones and draw these as boxes.
		This is related to the fact that interventions have a simple and intuitive description in terms of string diagram surgery~\cite{jacobs2021surgery}.
	\item More specifically, conditional independences are very easy to read off from the gs-monoidal string diagram. It amounts to first marginalizing over all outputs that one is not interested in (and applying a normalization), and then the conditional independence holds if and only if the diagram disconnects upon further removing all the wires that one conditions on~\cite[Remark~12.21]{fritz2020synthetic}.\footnote{We thank Rob Spekkens for sharing this observation with us.}
		This is in stark contrast with the high complexity of the $d$-separation criterion for Bayesian networks.
	\item Variables taking values in arbitrary measurable spaces, and continuous variables in particular, can be dealt with in the exact same way as discrete variables.
		The difference merely lies in the choice of a different target category, namely $\Stoch$ or $\BorelStoch$ as opposed to $\FinStoch$.
\end{itemize}
Of course, the first four features are merely about notation, since it is easy to introduce additional annotations on Bayesian networks. But this does not make them insignificant.
The last feature really is a feature of the Markov categories approach to probability in general, and is thus not specific to causal models.

In summary, our claims are that defining a generalized causal model as a morphism in a free Markov categories is feasible, and that this has a number of attractive features that the traditional notion of causal model does not have.
We leave more detailed investigations of generalized causal models in this sense to future work.
One may expect Patterson's work on statistical models in terms of Markov categories on the connections to categorical logic~\cite{patterson2020algebra} to be helpful in this regard.

\bibliographystyle{unsrt}
\bibliography{ref.bib}

\end{document}

%% file: markov2.tex
\tikzstyle{morphism}=[fill=white, draw=black, shape=rectangle]
\tikzstyle{medium box}=[fill=white, draw=black, shape=rectangle, minimum width=1.3cm, minimum height=0.7cm]
\tikzstyle{large morphism}=[fill=white, draw=black, shape=rectangle, minimum width=1.7cm]
\tikzstyle{bn}=[fill=black, draw=black, shape=circle, inner sep=1.5pt]
\tikzstyle{state}=[fill=white, draw=black, regular polygon, regular polygon sides=3, minimum width=0.8cm, shape border rotate=180, inner sep=0pt]
\tikzstyle{medium state}=[fill=white, draw=black, regular polygon, regular polygon sides=3, minimum width=1.3cm, inner sep=0pt, shape border rotate=180]
\tikzstyle{large state}=[fill=white, draw=black, regular polygon, regular polygon sides=3, minimum width=2.2cm, shape border rotate=180, inner sep=0pt]
\tikzstyle{wide state}=[fill=white, draw=black, shape=isosceles triangle, minimum width=0.8cm, shape border rotate=270, inner sep=1.4pt, minimum height=0.5cm, isosceles triangle apex angle=80]
\tikzstyle{wn}=[fill=white, draw=black, shape=circle, inner sep=1.5pt]
\tikzstyle{blue morphism}=[fill=white, draw={rgb,255: red,15; green,0; blue,150}, shape=rectangle, text={rgb,255: red,15; green,0; blue,150}, tikzit category=blue]
\tikzstyle{blue state}=[fill=white, draw={rgb,255: red,15; green,0; blue,150}, shape=circle, regular polygon, regular polygon sides=3, minimum width=0.8cm, shape border rotate=180, inner sep=0pt, text={rgb,255: red,15; green,0; blue,150}, tikzit category=blue]
\tikzstyle{blue node}=[fill={rgb,255: red,15; green,0; blue,150}, draw={rgb,255: red,15; green,0; blue,150}, shape=circle, tikzit category=blue, inner sep=1.5pt]
\tikzstyle{blue}=[text={rgb,255: red,15; green,0; blue,150}, tikzit draw={rgb,255: red,191; green,191; blue,191}, tikzit category=blue, tikzit fill=white, inner sep=0mm]
\tikzstyle{blue wide state}=[fill=white, draw={rgb,255: red,15; green,0; blue,150}, text={rgb,255: red,15; green,0; blue,150}, shape=isosceles triangle, minimum width=0.8cm, shape border rotate=270, inner sep=1.4pt, minimum height=0.5cm, isosceles triangle apex angle=80]
\tikzstyle{red node}=[fill={rgb,255: red,150; green,0; blue,2}, draw={rgb,255: red,150; green,0; blue,2}, shape=circle, inner sep=1.5pt]
\tikzstyle{Purple node}=[fill={rgb,255: red,150; green,0; blue,150}, draw={rgb,255: red,150; green,0; blue,150}, shape=circle, inner sep=1.5pt]
\tikzstyle{red}=[text={rgb,255: red,150; green,0; blue,2}, inner sep=0mm, tikzit fill=white, tikzit draw={rgb,255: red,191; green,191; blue,191}]
\tikzstyle{purple}=[text={rgb,255: red,150; green,0; blue,150}, inner sep=0mm, tikzit fill=white, tikzit draw={rgb,255: red,191; green,191; blue,191}]
\tikzstyle{white morphism}=[fill=white, draw=white, shape=rectangle, tikzit draw={rgb,255: red,139; green,139; blue,139}]
\tikzstyle{curly brace}=[decorate, decoration={brace,amplitude=5pt}]
\tikzstyle{costate}=[fill=white, draw=black, shape=circle, regular polygon, regular polygon sides=3, minimum width=0.8cm, inner sep=0pt]

\tikzstyle{arrow}=[->]
\tikzstyle{dashed box}=[-, dashed]
\tikzstyle{blue arrow}=[-, draw={rgb,255: red,15; green,0; blue,150}, tikzit category=blue]
\tikzstyle{mapsto}=[{|->}]
\tikzstyle{double wire}=[-, double]

%% file: copy.tex
\begin{tikzpicture}
	\begin{pgfonlayer}{nodelayer}
		\node [style=none] (0) at (0, 0.5) {};
		\node [style=bn] (1) at (0, 0.75) {};
		\node [style=none] (2) at (-0.5, 1.25) {};
		\node [style=none] (3) at (0.5, 1.25) {};
		\node [style=none] (4) at (0, 0.25) {$X$};
		\node [style=none] (5) at (-2.5, 0.75) {$\mathsf{copy}_X$};
		\node [style=none] (6) at (-1.25, 0.75) {=};
	\end{pgfonlayer}
	\begin{pgfonlayer}{edgelayer}
		\draw (0.center) to (1);
		\draw [bend left=45, looseness=1.25] (1) to (2.center);
		\draw [bend right=45, looseness=1.25] (1) to (3.center);
	\end{pgfonlayer}
\end{tikzpicture}

%% file: del.tex
\begin{tikzpicture}
	\begin{pgfonlayer}{nodelayer}
		\node [style=none] (5) at (0, 0.25) {};
		\node [style=bn] (6) at (0, 1) {};
		\node [style=none] (7) at (0, 0) {$X$};
		\node [style=none] (8) at (-2.25, 0.5) {$\mathsf{del}_X$};
		\node [style=none] (9) at (-1, 0.5) {=};
	\end{pgfonlayer}
	\begin{pgfonlayer}{edgelayer}
		\draw (5.center) to (6);
	\end{pgfonlayer}
\end{tikzpicture}

%% file: commutative_comonoid_eq.tex
\begin{tikzpicture}[scale=0.7]
	\begin{pgfonlayer}{nodelayer}
		\node [style=none] (0) at (-9, 1.25) {};
		\node [style=none] (1) at (-8, 1.25) {};
		\node [style=none] (2) at (-9, 0.25) {};
		\node [style=none] (3) at (-8, 0.25) {};
		\node [style=none] (5) at (-9, 0) {};
		\node [style=none] (6) at (-8, 0) {};
		\node [style=none] (8) at (-8.5, -1) {};
		\node [style=bn] (9) at (-8.5, -0.5) {};
		\node [style=none] (10) at (-7, 0) {$=$};
		\node [style=none] (11) at (-6, 1) {};
		\node [style=none] (12) at (-5, 1) {};
		\node [style=none] (14) at (-5.5, -1) {};
		\node [style=none] (15) at (-8.5, 0.75) {};
		\node [style=none] (16) at (-8.5, -1.25) {$X$};
		\node [style=none] (17) at (1, -1.25) {$X$};
		\node [style=bn] (18) at (-5.5, 0) {};
		\node [style=none] (19) at (-5.5, -1.25) {$X$};
		\node [style=none] (20) at (-1.5, -1.25) {$X$};
		\node [style=none] (22) at (-1, 1) {};
		\node [style=none] (23) at (-1.5, -1) {};
		\node [style=bn] (24) at (-1.5, 0) {};
		\node [style=none] (26) at (4.75, -1.25) {$X$};
		\node [style=bn] (27) at (-2, 1) {};
		\node [style=none] (28) at (0, 0) {$=$};
		\node [style=none] (29) at (1, 1) {};
		\node [style=none] (30) at (1, -1) {};
		\node [style=none] (31) at (9, -1.25) {$X$};
		\node [style=none] (32) at (4, 1) {};
		\node [style=none] (33) at (4.75, 1) {};
		\node [style=none] (34) at (5.75, 1) {};
		\node [style=bn] (35) at (5.25, 0.25) {};
		\node [style=bn] (36) at (4.75, -0.25) {};
		\node [style=none] (37) at (4.75, -1) {};
		\node [style=none] (38) at (6.75, 0) {$=$};
		\node [style=none] (39) at (9.75, 1) {};
		\node [style=none] (40) at (9, 1) {};
		\node [style=none] (41) at (8, 1) {};
		\node [style=bn] (42) at (8.5, 0.25) {};
		\node [style=bn] (43) at (9, -0.25) {};
		\node [style=none] (44) at (9, -1) {};
	\end{pgfonlayer}
	\begin{pgfonlayer}{edgelayer}
		\draw (2.center) to (5.center);
		\draw (3.center) to (6.center);
		\draw [bend left] (6.center) to (9);
		\draw [bend right, looseness=0.75] (5.center) to (9);
		\draw (9) to (8.center);
		\draw [bend left, looseness=0.75] (2.center) to (15.center);
		\draw [bend right, looseness=0.75] (3.center) to (15.center);
		\draw [bend right] (15.center) to (1.center);
		\draw [bend left] (15.center) to (0.center);
		\draw (14.center) to (18);
		\draw [bend left] (18) to (11.center);
		\draw [bend right] (18) to (12.center);
		\draw (23.center) to (24);
		\draw [bend right] (24) to (22.center);
		\draw [bend right] (27) to (24);
		\draw (30.center) to (29.center);
		\draw [bend right, looseness=0.75] (33.center) to (35);
		\draw [bend left, looseness=0.75] (34.center) to (35);
		\draw [bend right] (32.center) to (36);
		\draw [bend right=15] (36) to (35);
		\draw (37.center) to (36);
		\draw [bend left, looseness=0.75] (40.center) to (42);
		\draw [bend right, looseness=0.75] (41.center) to (42);
		\draw [bend left] (39.center) to (43);
		\draw [bend left=15] (43) to (42);
		\draw (44.center) to (43);
	\end{pgfonlayer}
\end{tikzpicture}

%% file: compatible_monoidal1.tex
\begin{tikzpicture}
	\begin{pgfonlayer}{nodelayer}
		\node [style=none] (46) at (-6.25, -0.5) {};
		\node [style=bn] (47) at (-6.25, 0.25) {};
		\node [style=none] (48) at (-7, 1) {};
		\node [style=none] (49) at (-5.5, 1) {};
		\node [style=none] (50) at (-4.5, 0) {=};
		\node [style=none] (51) at (-2.75, -0.5) {};
		\node [style=bn] (52) at (-2.75, 0.25) {};
		\node [style=none] (53) at (-3.5, 1) {};
		\node [style=none] (54) at (-2, 1) {};
		\node [style=none] (55) at (-2.25, -0.5) {};
		\node [style=bn] (56) at (-2.25, 0.25) {};
		\node [style=none] (57) at (-3, 1) {};
		\node [style=none] (58) at (-1.5, 1) {};
		\node [style=none] (59) at (2.25, -0.5) {};
		\node [style=bn] (60) at (2.25, 0.25) {};
		\node [style=none] (61) at (1.5, 1) {};
		\node [style=none] (62) at (3, 1) {};
		\node [style=none] (63) at (5, 0.75) {};
		\node [style=none] (64) at (5, -0.75) {};
		\node [style=none] (65) at (6.25, -0.75) {};
		\node [style=none] (66) at (6.25, 0.75) {};
		\node [style=none] (67) at (4, 0) {=};
		\node [style=none] (72) at (-6.25, -0.75) {$X \otimes Y$};
		\node [style=none] (73) at (-2.75, -0.75) {$X$};
		\node [style=none] (74) at (-2.25, -0.75) {$Y$};
		\node [style=none] (75) at (2.25, -0.75) {$I$};
	\end{pgfonlayer}
	\begin{pgfonlayer}{edgelayer}
		\draw [bend left=45, looseness=0.75] (47) to (48.center);
		\draw [bend right=45, looseness=0.75] (47) to (49.center);
		\draw (47) to (46.center);
		\draw [bend left=45, looseness=0.75] (52) to (53.center);
		\draw [bend right=45, looseness=0.75] (52) to (54.center);
		\draw (52) to (51.center);
		\draw [bend left=45, looseness=0.75] (56) to (57.center);
		\draw [bend right=45, looseness=0.75] (56) to (58.center);
		\draw (56) to (55.center);
		\draw [bend left=45, looseness=0.75] (60) to (61.center);
		\draw [bend right=45, looseness=0.75] (60) to (62.center);
		\draw (60) to (59.center);
		\draw [style=dashed box] (63.center) to (64.center);
		\draw [style=dashed box] (64.center) to (65.center);
		\draw [style=dashed box] (65.center) to (66.center);
		\draw [style=dashed box] (66.center) to (63.center);
	\end{pgfonlayer}
\end{tikzpicture}

%% file: compatible_monoidal2.tex
\begin{tikzpicture}
	\begin{pgfonlayer}{nodelayer}
		\node [style=none] (0) at (-6.25, -0.5) {};
		\node [style=none] (1) at (-4.5, 0) {=};
		\node [style=none] (2) at (-3, -0.5) {};
		\node [style=none] (3) at (-2, -0.5) {};
		\node [style=none] (4) at (2.25, -0.5) {};
		\node [style=none] (5) at (4, 0) {=};
		\node [style=bn] (6) at (-6.25, 0.5) {};
		\node [style=bn] (7) at (-3, 0.5) {};
		\node [style=bn] (8) at (-2, 0.5) {};
		\node [style=bn] (9) at (2.25, 0.5) {};
		\node [style=none] (10) at (5, 0.75) {};
		\node [style=none] (11) at (5, -0.75) {};
		\node [style=none] (12) at (6.25, -0.75) {};
		\node [style=none] (13) at (6.25, 0.75) {};
		\node [style=none] (14) at (-6.25, -0.75) {$X\otimes Y$};
		\node [style=none] (15) at (-3, -0.75) {$X$};
		\node [style=none] (16) at (-2, -0.75) {$Y$};
		\node [style=none] (17) at (2.25, -0.75) {$I$};
	\end{pgfonlayer}
	\begin{pgfonlayer}{edgelayer}
		\draw (6) to (0.center);
		\draw (7) to (2.center);
		\draw (8) to (3.center);
		\draw (9) to (4.center);
		\draw [style=dashed box] (10.center) to (11.center);
		\draw [style=dashed box] (11.center) to (12.center);
		\draw [style=dashed box] (12.center) to (13.center);
		\draw [style=dashed box] (13.center) to (10.center);
	\end{pgfonlayer}
\end{tikzpicture}

%% file: nat_counit.tex
\begin{tikzpicture}
	\begin{pgfonlayer}{nodelayer}
		\node [style=none] (0) at (-1, 0.25) {};
		\node [style=bn] (1) at (-1, 1.75) {};
		\node [style=none] (2) at (1, 0.25) {};
		\node [style=none] (5) at (0, 1) {=};
		\node [style=bn] (6) at (1, 1.75) {};
		\node [style=morphism] (7) at (-1, 1) {$f$};
	\end{pgfonlayer}
	\begin{pgfonlayer}{edgelayer}
		\draw (2.center) to (6);
		\draw (0.center) to (7);
		\draw (7) to (1);
	\end{pgfonlayer}
\end{tikzpicture}

%% file: monoidal_signature.tex
\begin{tikzpicture}
	\begin{pgfonlayer}{nodelayer}
		\node [style=none] (0) at (-0.5, -1) {};
		\node [style=none] (1) at (0, 1) {};
		\node [style=state] (2) at (-3, 0) {$p$};
		\node [style=none] (3) at (-3.25, 1) {};
		\node [style=none] (4) at (-2.75, 1.5) {$B$};
		\node [style=none] (5) at (0, 1.5) {$A$};
		\node [style=none] (6) at (-0.5, -1.5) {$A$};
		\node [style=none] (7) at (-3.25, 1.5) {$X$};
		\node [style=none] (8) at (-2.75, 1) {};
		\node [style=none] (9) at (0, -1) {};
		\node [style=none] (10) at (0, -1.5) {$A$};
		\node [style=none] (11) at (0.5, -1.5) {$B$};
		\node [style=none] (12) at (0.5, -1) {};
		\node [style=none] (13) at (-3.25, 0) {};
		\node [style=none] (14) at (-2.75, 0) {};
		\node [style=medium box] (15) at (0, 0) {$f$};
		\node [style=costate] (16) at (2.5, 0) {$g_1$};
		\node [style=none] (17) at (2.5, -1) {};
		\node [style=none] (18) at (2.5, -1.5) {$X$};
		\node [style=costate] (19) at (4, 0) {$g_2$};
		\node [style=none] (20) at (4, -1) {};
		\node [style=none] (21) at (4, -1.5) {$X$};
		\node [style=none] (22) at (5.5, 0) {$\cdots$};
		\node [style=none] (23) at (-0.5, -0.25) {};
		\node [style=none] (24) at (0, -0.25) {};
		\node [style=none] (25) at (0.5, -0.25) {};
	\end{pgfonlayer}
	\begin{pgfonlayer}{edgelayer}
		\draw [in=-90, out=90, looseness=1.25] (14.center) to (8.center);
		\draw [in=90, out=-90, looseness=1.25] (3.center) to (13.center);
		\draw (16) to (17.center);
		\draw (19) to (20.center);
		\draw (15) to (1.center);
		\draw (23.center) to (0.center);
		\draw (24.center) to (9.center);
		\draw (25.center) to (12.center);
	\end{pgfonlayer}
\end{tikzpicture}

%% file: cd_string_hypergraph.tex
\begin{tikzpicture}
	\begin{pgfonlayer}{nodelayer}
		\node [style=none] (0) at (0.25, -1.25) {};
		\node [style=none] (1) at (-0.5, 1.25) {};
		\node [style=none] (3) at (0.75, -1.5) {};
		\node [style=none] (4) at (-0.25, -3.25) {$A$};
		\node [style=none] (5) at (2, -2.75) {$B$};
		\node [style=none] (6) at (1.25, -1.25) {};
		\node [style=medium box] (7) at (0.75, -1) {$f$};
		\node [style=none] (8) at (0.25, -1.25) {};
		\node [style=none] (9) at (0.75, -1.25) {};
		\node [style=none] (10) at (1.25, -1.25) {};
		\node [style=state] (11) at (1.5, -3.5) {$p$};
		\node [style=none] (13) at (1.75, -2.75) {};
		\node [style=none] (14) at (1.25, -3.5) {};
		\node [style=none] (15) at (1.75, -3.5) {};
		\node [style=bn] (16) at (0, 3.25) {};
		\node [style=none] (17) at (-0.25, 2.75) {$A$};
		\node [style=none] (18) at (0, 1.25) {};
		\node [style=medium box] (20) at (0, 2) {$f$};
		\node [style=none] (21) at (-0.5, 1.75) {};
		\node [style=none] (22) at (0, 1.75) {};
		\node [style=none] (23) at (0.5, 1.25) {};
		\node [style=none] (24) at (0, -3.5) {};
		\node [style=none] (26) at (-0.75, 4) {};
		\node [style=none] (27) at (0.75, 4) {};
		\node [style=none] (28) at (0.75, -3) {$X$};
		\node [style=none] (29) at (0.25, -0.25) {$A$};
		\node [style=none] (31) at (0, -4.5) {$3$};
		\node [style=none] (32) at (-0.75, 4.5) {$1$};
		\node [style=none] (33) at (0.75, 4.5) {$2$};
		\node [style=none] (39) at (-2, -4.5) {$1$};
		\node [style=none] (40) at (-2, -4) {};
		\node [style=none] (41) at (-2.25, -3.25) {$A$};
		\node [style=bn] (42) at (-2, -2.5) {};
		\node [style=none] (43) at (-1, -4.5) {$2$};
		\node [style=none] (44) at (-1, -4) {};
		\node [style=none] (46) at (3.75, -4.5) {$4$};
		\node [style=none] (47) at (3.75, -4) {};
		\node [style=none] (49) at (3.25, -2.25) {$B$};
		\node [style=none] (50) at (-1.25, -3.25) {$X$};
		\node [style=bn] (51) at (-1, -2.75) {};
		\node [style=costate] (52) at (1, -2.5) {$g_2$};
		\node [style=none] (53) at (0.5, 1.75) {};
		\node [style=none] (54) at (0, -4) {};
	\end{pgfonlayer}
	\begin{pgfonlayer}{edgelayer}
		\draw [in=-90, out=90] (7) to (1.center);
		\draw (8.center) to (0.center);
		\draw (9.center) to (3.center);
		\draw (10.center) to (6.center);
		\draw [in=-90, out=90, looseness=1.25] (15.center) to (13.center);
		\draw (20) to (16);
		\draw (22.center) to (18.center);
		\draw [in=90, out=-90] (3.center) to (24.center);
		\draw [in=165, out=-90] (26.center) to (16);
		\draw [in=-90, out=15] (16) to (27.center);
		\draw (1.center) to (21.center);
		\draw (42) to (40.center);
		\draw [in=-90, out=15] (42) to (8.center);
		\draw [in=-90, out=165, looseness=0.75] (42) to (18.center);
		\draw [in=-90, out=90, looseness=0.50] (47.center) to (23.center);
		\draw [in=90, out=-90, looseness=1.50] (52) to (14.center);
		\draw (51) to (44.center);
		\draw [in=-90, out=90] (13.center) to (10.center);
		\draw (53.center) to (23.center);
		\draw (54.center) to (24.center);
	\end{pgfonlayer}
\end{tikzpicture}

%% file: noncd_string_hypergraph1.tex
\begin{tikzpicture}
	\begin{pgfonlayer}{nodelayer}
		\node [style=none] (0) at (0.25, -1.5) {};
		\node [style=none] (1) at (-0.5, 1) {};
		\node [style=none] (2) at (-1, -3.5) {$A$};
		\node [style=none] (3) at (0.75, -1.75) {};
		\node [style=none] (4) at (0, -3.5) {$A$};
		\node [style=none] (6) at (1.25, -1.5) {};
		\node [style=medium box] (7) at (0.75, -1.25) {$f$};
		\node [style=none] (8) at (0.25, -1.5) {};
		\node [style=none] (9) at (0.75, -1.5) {};
		\node [style=none] (10) at (1.25, -1.5) {};
		\node [style=bn] (21) at (0, 3) {};
		\node [style=none] (22) at (-0.25, 2.5) {$A$};
		\node [style=none] (23) at (0, 1) {};
		\node [style=medium box] (26) at (0, 1.75) {$f$};
		\node [style=none] (27) at (-0.5, 1.5) {};
		\node [style=none] (28) at (0, 1.5) {};
		\node [style=none] (29) at (0.5, 1.5) {};
		\node [style=none] (37) at (0.25, -4) {};
		\node [style=none] (38) at (-0.75, -1.75) {};
		\node [style=none] (40) at (0, 4) {};
		\node [style=none] (41) at (1, 4) {};
		\node [style=none] (43) at (0.25, -0.5) {$A$};
		\node [style=none] (45) at (-0.75, -4.5) {$1$};
		\node [style=none] (46) at (0.25, -4.5) {$2$};
		\node [style=none] (48) at (0, 4.5) {$1$};
		\node [style=none] (49) at (1, 4.5) {$2$};
		\node [style=none] (50) at (-1, 3.75) {};
		\node [style=none] (51) at (-1.5, -0.75) {};
		\node [style=none] (52) at (-0.75, -2.5) {};
		\node [style=none] (53) at (-0.75, -4) {};
		\node [style=none] (54) at (3.25, -4.5) {$3$};
		\node [style=none] (55) at (3.25, -4) {};
		\node [style=none] (56) at (3, -2.25) {$B$};
		\node [style=none] (57) at (2.25, -2.75) {$B$};
		\node [style=state] (58) at (1.75, -3.5) {$p$};
		\node [style=none] (60) at (2, -2.75) {};
		\node [style=none] (61) at (1.5, -3.5) {};
		\node [style=none] (62) at (2, -3.5) {};
		\node [style=none] (63) at (1, -3) {$X$};
		\node [style=costate] (64) at (1.25, -2.5) {$g_2$};
	\end{pgfonlayer}
	\begin{pgfonlayer}{edgelayer}
		\draw [in=-90, out=90] (7) to (1.center);
		\draw (8.center) to (0.center);
		\draw (9.center) to (3.center);
		\draw (10.center) to (6.center);
		\draw (26) to (21);
		\draw (28.center) to (23.center);
		\draw [in=90, out=-90] (3.center) to (37.center);
		\draw [in=90, out=-90, looseness=0.75] (40.center) to (21);
		\draw [in=-90, out=15] (21) to (41.center);
		\draw (1.center) to (27.center);
		\draw [in=90, out=-90] (23.center) to (38.center);
		\draw [in=180, out=90, looseness=0.50] (51.center) to (50.center);
		\draw [in=165, out=0] (50.center) to (21);
		\draw [in=-90, out=0] (52.center) to (8.center);
		\draw [in=-180, out=-90] (51.center) to (52.center);
		\draw (38.center) to (53.center);
		\draw [in=90, out=-90] (29.center) to (55.center);
		\draw [in=-90, out=90, looseness=1.25] (62.center) to (60.center);
		\draw [in=90, out=-90, looseness=1.50] (64) to (61.center);
		\draw [in=-90, out=90] (60.center) to (10.center);
	\end{pgfonlayer}
\end{tikzpicture}

%% file: noncd_string_hypergraph2.tex
\begin{tikzpicture}
	\begin{pgfonlayer}{nodelayer}
		\node [style=none] (0) at (0.25, -1.25) {};
		\node [style=none] (1) at (-0.5, 0.75) {};
		\node [style=none] (2) at (-1.25, -3.25) {$A$};
		\node [style=none] (3) at (0.75, -1.5) {};
		\node [style=none] (4) at (0, -3.25) {$A$};
		\node [style=none] (6) at (1.25, -1.25) {};
		\node [style=medium box] (7) at (0.75, -1) {$f$};
		\node [style=none] (8) at (0.25, -1.25) {};
		\node [style=none] (9) at (0.75, -1.25) {};
		\node [style=none] (10) at (1.25, -1.25) {};
		\node [style=bn] (16) at (0, 2.5) {};
		\node [style=none] (17) at (-0.25, 2) {$A$};
		\node [style=none] (18) at (0, 0.75) {};
		\node [style=medium box] (20) at (0, 1.25) {$f$};
		\node [style=none] (21) at (-0.5, 1) {};
		\node [style=none] (22) at (0, 1) {};
		\node [style=none] (23) at (0.5, 1) {};
		\node [style=none] (24) at (0.25, -3.75) {};
		\node [style=bn] (25) at (-1, -2.75) {};
		\node [style=none] (26) at (-0.75, 3.25) {};
		\node [style=none] (27) at (0.75, 3.25) {};
		\node [style=none] (31) at (0.25, -4.25) {$2$};
		\node [style=none] (32) at (-0.75, 3.75) {$1$};
		\node [style=none] (33) at (0.75, 3.75) {$2$};
		\node [style=none] (39) at (-1, -4.25) {$1$};
		\node [style=none] (40) at (-1, -3.75) {};
		\node [style=none] (41) at (2.25, -2.5) {$B$};
		\node [style=state] (42) at (1.75, -3.25) {$p$};
		\node [style=none] (44) at (2, -2.5) {};
		\node [style=none] (45) at (1.5, -3.25) {};
		\node [style=none] (46) at (2, -3.25) {};
		\node [style=none] (47) at (1, -2.75) {$X$};
		\node [style=costate] (48) at (1.25, -2.25) {$g_2$};
		\node [style=none] (49) at (3.75, -4.25) {$3$};
		\node [style=none] (50) at (3.75, -3.75) {};
		\node [style=none] (51) at (3.25, -2.25) {$B$};
		\node [style=bn] (52) at (-0.25, -0.25) {};
		\node [style=bn] (53) at (-0.25, 0.25) {};
	\end{pgfonlayer}
	\begin{pgfonlayer}{edgelayer}
		\draw (8.center) to (0.center);
		\draw (9.center) to (3.center);
		\draw (10.center) to (6.center);
		\draw (20) to (16);
		\draw (22.center) to (18.center);
		\draw [in=90, out=-90] (3.center) to (24.center);
		\draw [in=165, out=-90] (26.center) to (16);
		\draw [in=-90, out=15] (16) to (27.center);
		\draw (1.center) to (21.center);
		\draw [in=90, out=-90, looseness=0.75] (25) to (40.center);
		\draw [in=-90, out=90, looseness=1.25] (46.center) to (44.center);
		\draw [in=90, out=-90, looseness=1.50] (48) to (45.center);
		\draw [in=90, out=-90] (10.center) to (44.center);
		\draw [in=-90, out=90, looseness=0.50] (50.center) to (23.center);
		\draw [in=-165, out=165, looseness=0.75] (25) to (52);
		\draw [in=-90, out=15] (25) to (8.center);
		\draw [in=-15, out=90] (7) to (52);
		\draw [in=-90, out=165] (53) to (1.center);
		\draw [in=-90, out=15] (53) to (18.center);
		\draw (53) to (52);
	\end{pgfonlayer}
\end{tikzpicture}

%% file: bloom2.tex
\begin{tikzpicture}
	\begin{pgfonlayer}{nodelayer}
		\node [style=none] (0) at (0, -1) {};
		\node [style=bn] (1) at (-1, 1.25) {};
		\node [style=none] (2) at (0.5, -1.25) {};
		\node [style=none] (3) at (-0.5, -3.25) {$A$};
		\node [style=none] (4) at (3.25, -0.75) {$B$};
		\node [style=none] (5) at (1, -1) {};
		\node [style=medium box] (6) at (0.5, -0.75) {$f$};
		\node [style=none] (7) at (0, -1) {};
		\node [style=none] (8) at (0.5, -1) {};
		\node [style=none] (9) at (1, -1) {};
		\node [style=state] (10) at (1.75, -3.5) {$p$};
		\node [style=bn] (11) at (1.25, -2.75) {};
		\node [style=bn] (12) at (2.25, -2.75) {};
		\node [style=none] (13) at (1.5, -3.25) {};
		\node [style=none] (14) at (2, -3.25) {};
		\node [style=none] (16) at (-0.5, 3) {$A$};
		\node [style=none] (17) at (-0.25, 1.5) {};
		\node [style=medium box] (18) at (-0.25, 2.25) {$f$};
		\node [style=none] (19) at (-0.75, 2) {};
		\node [style=none] (20) at (-0.25, 2) {};
		\node [style=none] (21) at (0.25, 1.5) {};
		\node [style=bn] (22) at (-0.25, -2.25) {};
		\node [style=none] (24) at (-0.25, 4.25) {};
		\node [style=none] (25) at (1, -3) {$X$};
		\node [style=none] (26) at (0, 0) {$A$};
		\node [style=none] (27) at (-0.25, -4.5) {$3$};
		\node [style=none] (28) at (-4.5, 4.75) {$1$};
		\node [style=none] (29) at (-3.5, 4.75) {$2$};
		\node [style=none] (30) at (-2.25, -4.5) {$1$};
		\node [style=none] (31) at (-2.25, -4) {};
		\node [style=none] (32) at (-2.5, -3.25) {$A$};
		\node [style=bn] (33) at (-2.25, -2.25) {};
		\node [style=none] (34) at (-1.25, -4.5) {$2$};
		\node [style=none] (35) at (-1.25, -4) {};
		\node [style=none] (36) at (3.25, -4.5) {$4$};
		\node [style=bn] (37) at (3.25, -2.25) {};
		\node [style=none] (38) at (3, -3.25) {$B$};
		\node [style=none] (39) at (-1.5, -3.25) {$X$};
		\node [style=costate] (41) at (0.75, -2.25) {$g_2$};
		\node [style=none] (42) at (0.25, 2) {};
		\node [style=none] (43) at (-0.25, -4) {};
		\node [style=none] (44) at (-4.5, 4.25) {};
		\node [style=none] (45) at (-3.5, 4.25) {};
		\node [style=none] (46) at (-1.5, 4.75) {$4$};
		\node [style=none] (47) at (-1.5, 4.25) {};
		\node [style=none] (48) at (-2.5, 4.75) {$3$};
		\node [style=none] (49) at (-2.5, 4.25) {};
		\node [style=none] (50) at (-0.25, 4.75) {$5$};
		\node [style=none] (51) at (4.75, 4.25) {};
		\node [style=none] (52) at (4.75, 4.75) {$8$};
		\node [style=none] (53) at (3.25, -4) {};
		\node [style=none] (54) at (2.75, 4.25) {};
		\node [style=none] (55) at (2.75, 4.75) {$6$};
		\node [style=none] (56) at (3.75, 4.25) {};
		\node [style=none] (57) at (3.75, 4.75) {$7$};
		\node [style=none] (58) at (2, -3.5) {};
		\node [style=none] (59) at (1.5, -3.5) {};
	\end{pgfonlayer}
	\begin{pgfonlayer}{edgelayer}
		\draw [in=-90, out=90] (6) to (1);
		\draw (7.center) to (0.center);
		\draw (8.center) to (2.center);
		\draw (9.center) to (5.center);
		\draw [in=-90, out=90, looseness=1.25] (14.center) to (12);
		\draw [in=90, out=-90, looseness=1.75] (11) to (13.center);
		\draw (20.center) to (17.center);
		\draw [in=15, out=-90] (2.center) to (22);
		\draw [in=-90, out=15] (1) to (19.center);
		\draw (33) to (31.center);
		\draw [in=-90, out=15] (33) to (7.center);
		\draw [in=-90, out=90, looseness=0.75] (33) to (17.center);
		\draw [in=-90, out=165, looseness=0.50] (37) to (21.center);
		\draw [in=-90, out=165] (12) to (9.center);
		\draw (42.center) to (21.center);
		\draw (43.center) to (22);
		\draw (24.center) to (18);
		\draw [in=-90, out=165, looseness=0.50] (33) to (44.center);
		\draw [in=90, out=-90] (45.center) to (35.center);
		\draw [in=165, out=-90, looseness=0.50] (47.center) to (1);
		\draw [in=165, out=-90, looseness=0.25] (49.center) to (22);
		\draw [in=15, out=-90, looseness=0.50] (51.center) to (37);
		\draw (37) to (53.center);
		\draw [in=-90, out=15, looseness=0.25] (12) to (56.center);
		\draw [in=-75, out=165] (11) to (41);
		\draw [in=-90, out=15, looseness=0.25] (11) to (54.center);
		\draw (59.center) to (13.center);
		\draw (14.center) to (58.center);
	\end{pgfonlayer}
\end{tikzpicture}

%% file: effect.tex
\begin{tikzpicture}
	\begin{pgfonlayer}{nodelayer}
		\node [style=none] (0) at (0, -1) {};
		\node [style=costate] (1) at (0, 1) {$g_2$};
		\node [style=none] (2) at (-1, 0.25) {$\left(\rule{0cm}{1.6cm}\right.$};
		\node [style=none] (3) at (1, 0.25) {$\left.\rule{0cm}{1.6cm}\right)$};
		\node [style=none] (4) at (-2, 0.25) {$\norm$};
		\node [style=none] (5) at (2, 0.25) {$=$};
		\node [style=bn] (6) at (3, 1) {};
		\node [style=none] (7) at (3, -0.75) {};
	\end{pgfonlayer}
	\begin{pgfonlayer}{edgelayer}
		\draw (1) to (0.center);
		\draw (7.center) to (6);
	\end{pgfonlayer}
\end{tikzpicture}

%% file: markov_string_hypergraph.tex
\begin{tikzpicture}
	\begin{pgfonlayer}{nodelayer}
		\node [style=none] (0) at (-0.75, -1) {};
		\node [style=none] (1) at (-1.5, 1.5) {};
		\node [style=none] (2) at (-0.25, -1.25) {};
		\node [style=none] (3) at (-1.25, -3) {$A$};
		\node [style=none] (4) at (1, -2.5) {$B$};
		\node [style=none] (5) at (0.25, -1) {};
		\node [style=medium box] (6) at (-0.25, -0.75) {$f$};
		\node [style=none] (7) at (-0.75, -1) {};
		\node [style=none] (8) at (-0.25, -1) {};
		\node [style=none] (9) at (0.25, -1) {};
		\node [style=state] (10) at (0.5, -3.25) {$p$};
		\node [style=none] (11) at (0.75, -2.5) {};
		\node [style=none] (12) at (0.25, -3.25) {};
		\node [style=none] (13) at (0.75, -3.25) {};
		\node [style=none] (15) at (-1.25, 3) {$A$};
		\node [style=none] (16) at (-1, 1.5) {};
		\node [style=medium box] (17) at (-1, 2.25) {$f$};
		\node [style=none] (18) at (-1.5, 2) {};
		\node [style=none] (19) at (-1, 2) {};
		\node [style=none] (20) at (-0.5, 1.5) {};
		\node [style=none] (21) at (-1, -3.25) {};
		\node [style=none] (22) at (-2.25, 3.5) {};
		\node [style=none] (23) at (-1, 3.5) {};
		\node [style=none] (24) at (-0.25, -2.75) {$X$};
		\node [style=none] (25) at (-0.5, 0) {$A$};
		\node [style=none] (26) at (-1, -4.25) {$3$};
		\node [style=none] (27) at (-2.25, 4) {$1$};
		\node [style=none] (28) at (-1, 4) {$2$};
		\node [style=none] (29) at (-3, -4.25) {$1$};
		\node [style=none] (30) at (-3, -3.75) {};
		\node [style=none] (31) at (-3.25, -3) {$A$};
		\node [style=bn] (32) at (-3, -2.25) {};
		\node [style=none] (33) at (-2, -4.25) {$2$};
		\node [style=none] (34) at (-2, -3.75) {};
		\node [style=none] (35) at (2.75, -4.25) {$4$};
		\node [style=none] (36) at (2.75, -3.75) {};
		\node [style=none] (37) at (2.25, -2) {$B$};
		\node [style=none] (38) at (-2.25, -3) {$X$};
		\node [style=bn] (39) at (-2, -2.5) {};
		\node [style=bn] (40) at (0, -2.25) {};
		\node [style=none] (41) at (-0.5, 2) {};
		\node [style=none] (42) at (-1, -3.75) {};
		\node [style=bn] (43) at (-0.25, 0.25) {};
		\node [style=none] (44) at (1, 3.5) {};
		\node [style=none] (45) at (1, 4) {$3$};
	\end{pgfonlayer}
	\begin{pgfonlayer}{edgelayer}
		\draw (7.center) to (0.center);
		\draw (8.center) to (2.center);
		\draw (9.center) to (5.center);
		\draw [in=-90, out=90, looseness=1.25] (13.center) to (11.center);
		\draw (19.center) to (16.center);
		\draw [in=90, out=-90] (2.center) to (21.center);
		\draw (1.center) to (18.center);
		\draw (32) to (30.center);
		\draw [in=-90, out=15] (32) to (7.center);
		\draw [in=-90, out=90, looseness=0.75] (32) to (16.center);
		\draw [in=-90, out=90, looseness=0.50] (36.center) to (20.center);
		\draw [in=90, out=-90, looseness=1.50] (40) to (12.center);
		\draw (39) to (34.center);
		\draw [in=-90, out=90] (11.center) to (9.center);
		\draw (41.center) to (20.center);
		\draw (42.center) to (21.center);
		\draw (23.center) to (17);
		\draw [in=165, out=-90, looseness=0.75] (22.center) to (32);
		\draw (6) to (43);
		\draw [in=-90, out=165] (43) to (1.center);
		\draw [in=15, out=-90, looseness=0.75] (44.center) to (43);
	\end{pgfonlayer}
\end{tikzpicture}